\documentclass[12pt]{article}
\usepackage[numbers,square]{natbib}
\usepackage{amssymb,latexsym}
\usepackage[left=1in,top=1in,right=1in,bottom=1in]{geometry}

\usepackage[dvipsnames]{xcolor}
\definecolor{Bleu}{RGB}{0,0,204}
\definecolor{Violet}{RGB}{102,0,204}
\definecolor{Rouge}{RGB}{204,0,0}
\definecolor{Highlight}{RGB}{251,0,0}
\usepackage[hypertexnames=false]{hyperref}
\hypersetup{
colorlinks,
  citecolor=Bleu,
  linkcolor=Bleu,
  urlcolor=Violet} 

\usepackage{amsmath,amssymb}
\usepackage{amsthm,thmtools}

\usepackage{setspace}

\usepackage{enumitem}
\makeatletter
\newcommand{\mylabel}[2]{
	\addtocounter{\@listctr}{-1}%
    \refstepcounter{\@listctr}%
	#2\def\@currentlabel{#2}\label{#1}}
\makeatother

\usepackage{dsfont} 

\newtheorem{theorem}{Theorem}
\newtheorem{corollary}[theorem]{Corollary}
\newtheorem{lemma}[theorem]{Lemma}
\theoremstyle{definition}

\declaretheorem[name=Remark,qed={$\Box$},sibling=theorem]{remark} 

\newcommand{\bff}{{\bf f}}
\newcommand{\bP}{\mathbb{P}}
\DeclareMathOperator{\card}{card}
\DeclareMathOperator{\E}{\mathbb{E}}
\newcommand{\EP}{\textnormal{EP}}
\newcommand{\IF}{f}
\newcommand{\tIF}{\tilde{\IF}}
\DeclareMathOperator{\Ind}{\mathds{1}}
\newcommand{\pihat}{\widehat{\pi}}
\newcommand{\norm}[1]{\left\lVert#1\right\rVert}
\newcommand{\pibar}{\bar{\pi}}
\newcommand{\Pibar}{\overline{\Pi}}
\newcommand{\Rem}{\operatorname{Rem}}

\newcommand{\Reg}{\mathcal{R}}
\newcommand{\xA}{\mathcal{A}}
\newcommand{\xF}{\mathcal{F}}
\newcommand{\xG}{\mathbb{G}}
\newcommand{\xtF}{\tilde{\xF}}
\newcommand{\xL}{\mathcal{L}}
\newcommand{\xO}{\mathcal{O}}
\newcommand{\xP}{\mathcal{P}}
\newcommand{\xQ}{\mathcal{Q}}
\newcommand{\xR}{\mathbb{R}}
\newcommand{\xV}{\mathcal{V}}
\newcommand{\xVhat}{\widehat{\xV}}
\newcommand{\xX}{\mathcal{X}}
\newcommand{\xY}{\mathcal{Y}}

\allowdisplaybreaks

\usepackage[normalem]{ulem}
\usepackage{longtable}
\usepackage{authblk}

\title{Faster Rates for Policy Learning} 

\author[1,2]{Alexander R.  Luedtke}

\affil[1]{\footnotesize   Vaccine  and   Infectious  Disease   Division,  Fred
  Hutchinson Cancer Research Center, USA}
\affil[2]{\footnotesize   Public Health Sciences  Division,  Fred
  Hutchinson Cancer Research Center, USA} 

\author[3,4]{Antoine Chambaz}

\affil[3]{\footnotesize  Modal'X,  UPL, Univ  Paris
  Nanterre, F92000 Nanterre France}

\affil[4]{\footnotesize Division of Biostatistics, School of Public Health, UC
  Berkeley, USA}

\date{}

\begin{document}
\maketitle

\begin{abstract}
  This article improves  the existing proven rates of regret  decay in optimal
  policy estimation.   We give  a margin-free result  showing that  the regret
  decay  for estimating  a  within-class optimal  policy  is second-order  for
  empirical risk  minimizers over Donsker  classes, with regret decaying  at a
  faster rate than  the standard error of an efficient  estimator of the value
  of  an optimal  policy.   We  also give  a  result  from the  classification
  literature  that shows  that faster  regret  decay is  possible via  plug-in
  estimation provided a margin condition holds.  Four examples are considered.
  In these examples, the regret is expressed in terms of either the mean value
  or  the median  value;  the number  of  possible actions  is  either two  or
  finitely many; and the sampling scheme is either independent and identically
  distributed or sequential,  where the latter represents  a contextual bandit
  sampling scheme.
\end{abstract}

\noindent%
 \vfill

\section{Introduction}

\subsection{Objective}

We consider  an experiment where a  player repeatedly chooses and  carries out
one  among two  actions to  receive a  random reward.   Each time,  the reward
depends on both the action undertaken and a random context preceding it, which
is given to the player before she makes her choice. The law of the context and
conditional  law  of  the  reward  given the  context  and  action  are  fixed
throughout the  experiment.  The player's  objective is  to obtain as  large a
cumulated sum of rewards as possible.

In this framework, a policy is a rule that maps any context to an action.  The
value of a policy is the expectation of the reward in the experiment where the
action carried  out is  the action  recommended by the  policy. Given  a class
$\Pi$  of candidate  policies,  the  regret of  a  policy  $\pi\in\Pi$ is  the
difference between the largest value achievable  within $\Pi$ and the value of
$\pi$.

Learning the optimal policy within a class of candidate policies is meaningful
whenever the goal is to make recommendations.  This is, for instance, the case
in personalized medicine, also known as precision medicine. There, the context
would typically  consist of the  description of  a patient, the  actions would
correspond to two strategies of treatment,  and the policies are rather called
individualized treatment rules.

The objective  of this article is  not to establish optimal  regret bounds for
optimal  policy estimators.  It is,  rather, to  show that  rates faster  than
$n^{-1/2}$ can  be demonstrated under  much more general conditions  than have
previously been discussed in the policy learning literature.

\subsection{A Brief Literature Review}

There  has  been a  surge  of  interest  in  developing flexible  methods  for
estimating optimal policies in recent years. Here we give a deeply abbreviated
overview,  and  refer  the  reader to  \citep{Athey&Wager2017}  for  a  recent
overview of the literature.  Exciting developments in policy learning over the
last     several      years     include     outcome      weighted     learning
\citep{Zhaoetal12,Zhangetal12,Rubin&vanderLaan12},     ensemble     techniques
\citep{Luedtke&vanderLaan2014},   and   empirical   risk   minimizers   (ERMs)
\citep{Athey&Wager2017}, to name  a few.  Each of these  works has established
some form of regret  bound for the estimator, with most  of these regret rates
slower  than $n^{-1/2}$.   Notable exceptions  to this  $n^{-1/2}$ restriction
occur  for  outcome  weighted  learning  (OWL) methods  under  a  hard  margin
\citep{Zhaoetal12}  and  rates  attainable by  plug-in  estimation  strategies
\citep{Farahmand2011,Luedtke&vanderLaan2015b,Chambazetal2017}.

We  give  our  results  in Section~\ref{sec:mainresults}.   Our  first  result
pertains  to empirical  risk minimization  within a  Donsker class,  where the
optimal policy  is defined  as the so-called  ``value'' maximizer  within this
class. These estimators were  recently studied in \citep{Athey&Wager2017}, and
were      also      discussed      for      continuous      treatments      in
\citep{Luedtke&vanderLaan2016b}. The  second result  links the  optimal policy
problem to the  classification literature, which has shown  that faster regret
decay  rates  are  attainable  by  plug-in  estimators  under  certain  margin
conditions \citep{Audibert&Tsybakov2007}. Section~\ref{sec:threeremarks} gives
three  remarks,  one  relating  our  results  to  the  pioneering  results  of
\citet{Koltchinskii2006},  another studying  which  term in  our regret  bound
dominates  the  rate,  and  the  third   relating  our  results  to  those  of
\citet{Athey&Wager2017}.  Section~\ref{sec:proofregDonsker} proves our results
regarding empirical risk minimization.

All  of  our  results  are  for  the regret  under  a  fixed  data  generating
distribution. The  fast rate that we  will establish for ERMs  (independent of
any margin  assumption) will not transfer  to the minimax setting  unless some
form of margin assumption is imposed. We do not give high probability results,
though  these  are  often  a   straightforward  extension  of  convergence  in
probability results when working within Donsker classes. Throughout this work,
when we  refer to Donsker  classes, we  are referring to  $P$-Donsker classes,
where $P$ is  the data generating distribution,  rather than universal/uniform
Donsker  classes \citep{Sheehy&Wellner1992}.   While our  ERM results  are not
useful for non-$P$-Donsker classes, we note  that they also serve as an oracle
guarantee for an ensemble procedure, such as \cite{Luedtke&vanderLaan2014}, so
that one need not know in advance whether or not all of the classes over which
they are optimizing are $P$-Donsker for the result to be useful.

We do not concern ourselves  with measurability issues, with the understanding
that   some   work  may   be   needed   to   make  these   arguments   precise
\citep{vanderVaartWellner1996}.

\subsection{Formalization}\label{sec:formalization}

Let $X \in \xX$ denote a vector of covariates describing the context preceding
the action, $A \in \{-1,1\}$ denote  the action undertaken, and $Y$ denote the
corresponding reward, where here larger  rewards are preferable.  We denote by
$P$ the  distribution of $O\equiv (X,A,Y)$  and by $\E$ the  expectation under
$P$.  For  simplicity we assume throughout  that, under $P$, $Y$  is uniformly
bounded and  that there exists  some $\delta  \in (0,1/2)$ so  that $P(A=1|X)$
falls in $[\delta,1-\delta]$  with probability one.  The  second assumption is
known as the strong positivity assumption.  While stronger than we need, these
assumptions simplify our analysis.

Let $\xP$ be the  class of all policies, the subset  of $L^2(P)$ consisting of
functions mapping $\xX$ to $\{-1,1\}$.  The  value of a policy $\pi\in \xP$ is
given by
\begin{equation}
  \label{eq:def:value}
  \xV(\pi)\equiv \E\left[\E[Y|A=\pi(X),X]\right].
\end{equation}
Under some causal assumptions that we will not explore here, $\xV(\pi)$ can be
identified with the mean reward if, possibly contrary to fact, action $\pi(X)$
is carried out in context $X$.

Now, let  $\Pi \subset \xP$  be the class  of candidate policies.   The regret
(within class $\Pi$) of $\pi$ is  defined as the difference between $\xV(\pi)$
and  the optimal  value  $\xV^\star\equiv  \sup_{\pi\in\Pi}\xV(\pi)$: for  all
$\pi \in \Pi$,
\begin{equation}
  \label{eq:def:regret}
  \Reg(\pi)\equiv \xV^\star - \xV(\pi).
\end{equation}
We extend  the definition of $\Reg$  to $\xP$.  Obviously, $\Reg(\pi)  \geq 0$
for all $\pi \in \Pi$, but $\Reg$ is not necessarily nonnegative over $\xP$.

For every $\pi \in \xP$, $\xV(\pi)$ can  be viewed as the evaluation at $P$ of
the functional
\begin{equation*}
  P'  \mapsto \E_{P'}\left[\E_{P'}[Y|A=\pi(X),X]\right]
\end{equation*}
from  the  nonparametric  model  of distributions  $P'$  satisfying  the  same
constraints  as   $P$  to  the   real  line.   This  functional   is  pathwise
differentiable at $P$ relative to the  maximal tangent space with an efficient
influence function $\IF_{\pi}$ given by
\begin{equation}
  \label{eq:IF:pi}
  \IF_{\pi}(o) \equiv 
  \frac{\Ind\{a=\pi(x)\}}{P(A=a|X=x)}\left(y-\E[Y|A=a,X=x]\right)            +
  \E[Y|A=\pi(x),X=x] - \xV(\pi). 
\end{equation}
By  the bound  on $Y$  and the  strong positivity  assumption, there  exists a
universal            constant            $M>0$            such            that
$ P(\sup_{\pi \in \Pi}|\IF_{\pi}(O)|\leq M)=1$.

We   refer  the   interested  reader   to~\cite[][Section~25.3]{vanderVaart98}
and~\cite{Luedtke&vanderLaan2015b} for definitions and  proofs of these facts.
Although the class
\begin{equation}
  \label{eq:IF}
  \xF  \equiv \{\IF_{\pi}  :  \pi \in  \Pi\} \subset  L^{2}(P)
\end{equation}
will play a central  role in the rest of this article, it  is not necessary to
master  the  derivation or  properties  of  efficient influence  functions  to
appreciate the contributions of this work.

\section{Main Results}\label{sec:mainresults}

\subsection{Preliminary}
\label{subsec:prelim}

Suppose    that    we    observe    mutually    independent    data-structures
$O_{1} \equiv(X_1,A_1,Y_1), \ldots,  O_{n}\equiv(X_n,A_n,Y_n)$ drawn from $P$.
Let $P_n \equiv n^{-1}\sum_{i=1}^{n} \text{Dirac}(O_{i})$ denote the empirical
measure.  For every  function $f$ mapping an observation to  the real line, we
set             $P            f\equiv             \E[f(X,A,Y)]$            and
$P_{n} f \equiv n^{-1} \sum_{i=1}^{n} f(O_{i})$.

Our theorem will attain fast rates under the assumption that one has available
an     estimator    $\{\xVhat(\pi)     :    \pi\in\Pi\}$     of
$\{\xV(\pi) : \pi\in\Pi\}$ that makes the following term small:
\begin{equation}
  \Rem_n\equiv     \sup_{\pi\in\Pi}     \left|\xVhat(\pi)     -
    \xV(\pi) - (P_n-P) \IF_{\pi}\right|. \label{eq:valest} 
\end{equation}

If empirical process  conditions are imposed on the estimators  for the reward
regression,  $\E[Y|A,X]$, and  the action  mechanism, $P(A|X)$,  then $\xVhat$
could be defined using estimating  equations \citep{vdL02} or targeted minimum
loss-based  estimation   (TMLE)  \citep{vanderLaan&Rubin06,vanderLaan&Rose11}.
Without imposing  empirical process conditions, one  could use cross-validated
estimating   equations  \citep{vanderLaan&Luedtke2016,vanderLaan&Luedtke2014},
now  also called  double machine  learning \citep{Chernozhukovetal2016},  or a
cross-validated  TMLE \citep{Zheng&vanderLaan11}.   Note that  cross-validated
estimating  equations  in  this  scenario   represent  a  special  case  of  a
cross-validated  TMLE, where  one uses  the squared  cross-validated efficient
influence function  as loss \citep{Chaffee&vanderLaan2011}.  For  any of these
listed   approaches,   obtaining   the   above  uniformity   over   $\pi$   is
straightforward if $\Pi$ is a Donsker class.

\subsection{Empirical Risk Minimizers}

Our first  objective will  be to  establish a faster  than $n^{-1/2}$  rate of
regret decay for a value-based  estimator $\pihat$ of an optimal policy,
given by any ERM $\pihat \in \Pi$ satisfying
\begin{equation}
  \xVhat(\pihat)\ge \sup_{\pi\in\Pi} \xVhat(\pi) - o_P(\Rem_n).\label{eq:valmax}
\end{equation}
Note  that \eqref{eq:valmax}  is a  requirement  on the  behavior of  the
  optimization algorithm on the realized  sample, rather than a statistical or
  probabilistic condition.   The $o_P(\Rem_n)$  term allows  us to  obtain an
approximate  solution  to  the  ERM  problem,  which  is  useful  because  the
optimization on the right is non-concave.

We now present a result that is similar to the groundbreaking results for ERMs
based   on    general   losses    given   in    \cite{Koltchinskii2006}.    In
Section~\ref{rem:relaterates}, we discuss  how our results relate  to those in
\cite{Koltchinskii2006}.

The  upcoming theorem  uses  $\Pibar$ to  denote the  closure  of $\Pi$  under
$L^2(P)$ norm.  The set
\begin{eqnarray}
  \label{eq:Pibar}
  \Pibar^\star 
  &\equiv& \{\pi^\star\in\Pibar : \xV(\pi^\star)=\xV^\star\}\\
  \notag
  &=& \{\pi^{\star} \in \Pibar : \Reg(\pi^{\star}) = 0\} \subset \{\pi^{\star}
      \in \Pibar : \Reg(\pi^{\star}) \leq 0\} 
\end{eqnarray}
also   plays   an   important   role.   If   $\Pi$   is   Donsker,   then
  Lemma~\ref{lem:Pistarnonempty}      in     Section~\ref{sec:proofregDonsker}
  guarantees that $\Pibar^\star$ is nonempty and coincides with the RHS set in
  the above display.  Finally, we define
\begin{equation*}
  \EP_n\equiv    \liminf_{s\downarrow    0}    \inf_{\pi^\star\in\Pibar^\star}
  \sup_{\pi\in            \Pi             :            \norm{\pi-\pi^\star}\le
    s}(P_n-P)(\IF_{\pihat}-\IF_{\pi}), 
\end{equation*}
which roughly  corresponds to the best  approximation in $\Pi$ of  the closest
optimal policy  $\pi^\star\in\Pibar^\star$ to  the estimated  policy $\pihat$.
We will show that $\EP_n=o_P(n^{-1/2})$ under mild assumptions.
\begin{theorem}[Main ERM Result]\label{thm:regDonsker}
  If $\pihat$ satisfies (\ref{eq:valmax}) and $\Pi$    is    Donsker, then
  \begin{equation*}
    0 \leq \Reg(\pihat)\le \EP_n + [1+o_P(1)]\Rem_n.
  \end{equation*}
  If also $\Rem_n=o_P(n^{-1/2})$,    then
  $\Reg(\pihat) = o_P(n^{-1/2})$.
\end{theorem}
The \hyperlink{proof:thmregDonsker}{proof of Theorem~\ref*{thm:regDonsker}} is
given in Section~\ref{sec:proofregDonsker}.  The  above rate on $\Reg(\pihat)$
is faster than the  rate of convergence of the standard  error of an efficient
estimator of the value $\xV(\pi)$ of  a policy $\pi\in\Pi$, which converges at
rate $n^{-1/2}$ in all but degenerate cases.  Note that our proof of the first
result  of the  above  theorem only  uses  that $\Pi$  is  totally bounded  in
$L^2(P)$ while that of the second  one uses the stronger assumption that $\Pi$
is   Donsker.   Three   detailed  remarks   on  this   result  are   given  in
Section~\ref{sec:threeremarks}.

\subsection{Plug-In Estimators}
Section~\ref{rem:relaterates} discusses  why we  do not  expect a  faster than
$O_P(n^{-1})$  rate of  regret decay  for $\pihat$,  even within  a parametric
model. A notable exception to this rule  of thumb occurs if $\Pi$ is of finite
cardinality and $P(A|X)$  is known, since in this case  large deviation bounds
suggest very fast rates of convergence for ERMs.  The same phenomenon has been
noted and extensively studied  in the classification literature~\citep[see the
review in][]{Tsybakov04}.  

We now show that faster rates are attainable under some conditions if one uses
a  different  estimation  procedure.   We  point  this  out  because,  in  our
experience, this  alternative estimation  strategy can sometimes  yield better
estimates than a value-based strategy \citep{Luedtke&vanderLaan2014}.

Let  $\gamma(X)\equiv  \E[Y|A=1,X]  -  \E[Y|A=-1,X]$  denote  the  conditional
average  action  effect. In  this  subsection,  we  assume  that $\Pi$  is  an
unrestricted  class, {\em  i.e.}, that  $\Pi$ contains  all functions  mapping
$\xX$ to  $\{-1,1\}$.  \citet{Audibert&Tsybakov2007} presented  the surprising
fact  that plug-in  classifiers  can  attain much  faster  rates (faster  than
$n^{-1}$). In our setting, the plug-in policy that we will study first defines
an estimator $\widehat{\gamma}$ of $\gamma$, and then determines the action to
undertake based on the sign of  the resulting estimate.  Formally, $\pihat$ is
given by
  \begin{equation*}
    \pihat(x)\equiv                \Ind\{\widehat{\gamma}(x)>0\}               -
    \Ind\{\widehat{\gamma}(x)\leq 0\}. 
  \end{equation*}
Extensions of the result  of \cite{Audibert&Tsybakov2007} have previously been
presented in both the  reinforcement learning literature \citep{Farahmand2011}
and     in     the     optimal     individualized     treatment     literature
\citep{Luedtke&vanderLaan2015b,Chambazetal2017}. We therefore  omit any proof,
and refer the reader  to~\cite[][Lemma 5.2]{Audibert&Tsybakov2007} for further
details.

The  upcoming theorem  uses $\norm{\cdot}$  to  denote the  $L^2(P)$ norm  and
$\norm{\cdot}_{\infty}$  to denote  the  $L^\infty(P)$  norm.  Let  $\lesssim$
denote  ``less than  or equal  to up  to a  universal positive  multiplicative
constant'' and consider the following ``margin assumption'':
\begin{enumerate}
\item[\mylabel{it:ma}{MA)}]           For           some           $\alpha>0$,
  $P(0<|\gamma(X)|\le t)\lesssim t^{\alpha}$ for all $t>0$.
\end{enumerate}

\begin{theorem}
  Suppose             \hyperref[it:ma]{MA}              holds.              If
  $\norm{\widehat{\gamma}-\gamma}=o_P(1)$,                                then
  $0                         \leq                         \Reg(\pihat)\lesssim
  \norm{\widehat{\gamma}-\gamma}^{2(1+\alpha)/(2+\alpha)}$.                 If
  $\norm{\widehat{\gamma}-\gamma}_{\infty}=o_P(1)$,                       then
  $0                         \leq                         \Reg(\pihat)\lesssim
  \norm{\widehat{\gamma}-\gamma}_{\infty}^{1+\alpha}$.
\end{theorem}

We now briefly describe this result, though we note that it has been discussed
thoroughly                                                           elsewhere
\citep{Audibert&Tsybakov2007,Luedtke&vanderLaan2015b,Chambazetal2017}.    Note
that  \hyperref[it:ma]{MA} does  not  place any  restriction  on the  decision
boundary $\{x  \in \xX  : \gamma(x)=0\}$  where no action  is superior  to the
other,  but  rather  only  places   a  restriction  on  the  probability  that
$\gamma(X)$ is near zero.

If  $\alpha=1$ and  $\gamma(X)$  is absolutely  continuous  (under $P$),  then
\hyperref[it:ma]{MA} corresponds  to $\gamma(X)$  having bounded  density near
zero. The  case that  $\alpha=1$ is  of particular  interest for  the sup-norm
result because then  the regret bound is quadratic in  the rate of convergence
of $\widehat{\gamma}$  to $\gamma$.   As $\alpha\rightarrow  0$, more  mass is
placed near the decision boundary (zero)  and the above result yields the rate
of     convergence     of      $\widehat{\gamma}$     to     $\gamma$.      As
$\alpha\rightarrow \infty$, a vanishingly small amount of mass is concentrated
near zero  and the above result  recovers very fast rates  of convergence when
$\widehat{\gamma}$ converges to $\gamma$ uniformly.

\begin{remark}[Estimating $\gamma$]
  Doubly  robust unbiased  transformations \citep{Rubin&vanderLaan06}  provide
  one way to estimate $\gamma$ (Section 3.1 of \citep{vanderLaan&Luedtke14b}).
  In   particular,  one   can  regress   (via  any   desired  algorithm)   the
  pseudo-outcome
  \begin{equation*}
    \widehat{\Gamma}(O)\equiv
    \frac{A}{\widehat{P}(A|X)}          \left(Y         -
      \widehat{\E}[Y|A,X]\right)        +       \widehat{\E}[Y|A=1,X]        -
    \widehat{\E}[Y|A=-1,X] 
  \end{equation*}
  against $X$,  both for  double robustness and  for efficiency  gains.  Above
  $\widehat{P}(A|X)$   and  $\widehat{\E}[Y|A,X]$   are  estimators   of
  $P(A|X)$ and $\E[Y|A,X]$.  One can  use cross-validation to avoid dependence
  on empirical  process conditions  for the  estimators of  $P(A|X)$ and
  $\E[Y|A,X]$.
\end{remark}

\section{Three                            Remarks                           on
  Theorem~\ref{thm:regDonsker}}\label{sec:threeremarks}

\subsection{Relation            to           the            Rates           of
  \citet{Koltchinskii2006}}\label{rem:relaterates}

Our         Theorem~\ref{thm:regDonsker}         is         related         to
\cite[][Theorem~4]{Koltchinskii2006} and  to the discussion  of classification
problems in Section~6.1  of this landmark article, although we  {\em (i)} make
no  assumptions  on  the  behavior  of $\gamma$  near  the  decision  boundary
(``margin assumptions''),  {\em (ii)} give  a fixed-$P$ rather than  a minimax
result, {\em  (iii)} do not require  that the regret minimizer  belongs to the
set $\Pi$, and {\em (iv)} make a slightly weaker complexity requirement on the
class $\Pi$  (only requiring $\Pi$  Donsker).  We thus have  weaker conditions
(only constrain  the complexity  of $\Pi$ using  a general  Donsker condition)
and, consequently, a weaker implication.  Our result also differs from that of
\cite{Koltchinskii2006} due to the need for  us to control the extra remainder
term arising from \eqref{eq:valest}.

Our          Lemmas~\ref{lem:cont}         and~\ref{lem:gsemicont}          in
Section~\ref{sec:proofregDonsker}  are key  to giving  this general  ``Donsker
implies fast rate'' implication, even when  $\Pi$ does not achieve the infimum
on the regret.  Once this crucial lemma  is established and one has dealt with
the existence of  the $\Rem_n$ remainder term, one could  use simlar arguments
to those  used in  \cite[][Section~4]{Koltchinskii2006} to control  the regret
(though, the proofs in our work are self-contained).  We were not able to find
a  similar result  in  the  classification literature  that  shows that  $\Pi$
Donsker  suffices  to attain  a  fast  rate  on misclassification  error  (the
classification analogue of our regret).  A  result that makes a similarly weak
complexity requirement (only  uses a Donsker condition) was  given for general
result for ERMs in \cite[][Theorem 4.5]{Koltchinskii2009}. In particular, that
result shows that, if $\Pi$ is Donsker and $\Pibar^\star$ is a singleton, then
one attains  a fast rate.   In the policy learning  context, this is  a weaker
result than what we have shown: $\Pi$ Donsker implies a fast rate, without any
assumption on the cardinality of $\Pibar^\star$.  We note also that our result
could readily be  extended to standard classification problems: the  lack of a
remainder term $\Rem_n$ only makes the problem easier.

\subsection{Tightening Theorem~\ref{thm:regDonsker}}\label{rem:tightening}

Tightening Theorem~\ref{thm:regDonsker} would require careful consideration of
the  rate of  convergence of  two $o_P(n^{-1/2})$  terms, namely  $\Rem_n$ and
$\EP_n$.  On the one hand, if one uses a cross-validated estimator, then
the rate  of $\Rem_n$ is  typically dominated by the  rate of a  doubly robust
term, which is in turn upper bounded by the product of the $L^2(P)$ norm
rates of convergence of the  estimated action mechanism and reward
regression. It  is thus  sufficient  to  assume  that this  product  is
  $o_P(n^{-1/2})$  to  guarantee  that $\Rem_n=o_P(n^{-1/2})$.   It  is  worth
  noting  that the  product is  $O_P(n^{-1})$ in  a well-specified  parametric
  model.  On the other hand, the  magnitude of $\EP_n$ is controlled by
both the size  of class $\Pi$ and  the behavior of $\gamma$  near the decision
boundary,    hence   the    need   for    a   margin    assumption   like
  \hyperref[it:ma]{MA}. 

We do not believe there is a  general ordering between the rate of convergence
of $\Rem_n$ and $\EP_n$  that applies across all problems, and  so it does not
appear that the size of $\Pi$ nor  the use of an efficient choice of influence
function fully  determines the  rate of  regret decay  of $\pihat$,  even when
$\Rem_n=o_P(n^{-1/2})$.  A notable  exception to this lack  of strict ordering
between $\Rem_n$ and $\EP_n$ occurs  when the action mechanism is known:
in  this  case,  the size  of  $\Pi$  and  the  behavior of  $\gamma$  on  the
decision  boundary fully  control the  rate of  regret decay  whenever a
cross-validated estimator is used.  We also note  that, as far as we can tell,
there does  not appear  to be any  cost to using  an efficient  value function
estimator  when the  model is  nonparametric.  It  is not  clear if  using the
efficient  influence  function  is   always  preferred  in  more  restrictive,
semiparametric models:  there may need  to be  a careful tradeoff  between the
efficiency of the influence function and the corresponding $\Rem_n$.

\subsection{Relation to Results of \citet{Athey&Wager2017}}\label{sec:AtheyWager}

In       the       recent       technical       report~\cite{Athey&Wager2017},
\citeauthor{Athey&Wager2017} showed  that policy learning regret  rates on the
order  of $O_P(n^{-1/2})$  are  attainable by  ERMs.  High-probability  regret
upper bounds were derived, with leading constants that scale with the standard
error  of a  semiparametric efficient  estimator for  policy evaluation.   The
authors argue that this leading  constant demonstrates the importance of using
semiparametric efficient value estimators to define the empirical risk used by
their estimator.   As we  discussed in Section~\ref{rem:tightening},  we fully
agree  with the  importance  of  using efficient  estimators  to estimate  the
empirical risk.  Nonetheless,  the present work shows that the  regret of ERMs
decays faster than  the standard error of an efficient  estimator of the value
and so, the regret  of ERMs that use this empirical  risk will not necessarily
scale with the standard error of this estimator.

Like  in~\citep{Athey&Wager2017}, our  results are  given under  a fixed  data
generating distribution $P$.  We implicitly  leverage the behavior of $\gamma$
near the  decision boundary  (zero) under  this distribution  $P$.  Crucially,
there is  a problem-dependent constant that  can be made arbitrarily  large if
one chooses  a $P$ for which  $\gamma(X)$ concentrates a large  amount of mass
near  the decision  boundary.   Hence, the  minimax rate  is  not faster  than
$n^{-1/2}$ unless one  constrains the class of distributions to  which $P$ can
belong via  a margin condition (see  \citep{Tsybakov04,Koltchinskii2006}). The
implication of this observation is encouraging:  it would not be surprising to
find that, without  a margin condition, the minimax regret  does in fact decay
at the  rate of  the standard error  of an efficient  estimator of  an optimal
policy.   The leading  constant may  also critically  depend on  the efficient
variance  of   such  an   estimator.   It  is   worth  studying   whether  the
high-probability,  finite sample  results  in  \citep{Athey&Wager2017} can  be
extended  in  this direction,  thereby  demonstrating  a criterion  for  which
efficient value  estimation is mathematically indispensable  for obtaining the
optimal regret decay.

\section{Extension of Main ERM Result and Three Examples}\label{sec:extension}

\subsection{Higher Level Result}
\label{subsec:higher}

Suppose we observe $(O_1,\ldots,O_n)$ drawn from a distribution $\nu_n$, where
each     $O_i     \equiv     (X_i,A_i,Y_i)$    takes     its     values     in
$\xO \equiv  \xX \times \xA  \times \xY$ with  $\xA \subset [-1,1]$.   Like in
Section~\ref{sec:formalization}, $X_{i}\in \xX$ denotes a vector of covariates
describing the context preceding the $i$th action, $A_{i} \in \xA$ denotes the
action undertaken  in this context  and $Y_{i}  \in \xY$ is  the corresponding
reward.  Unlike in Section~\ref{sec:formalization}, there may be more than two
actions and the  observations are not necessarily  independent and identically
distributed   (i.i.d.).     This   extends    the   setting    introduced   in
Section~\ref{sec:formalization}, which can be  recovered with $\xA = \{-1,1\}$
and $\nu_n$ equal to a product  measure.  Throughout we also assume that there
exists a distribution $P$ with support on  $\xO$ that will have to do with our
limit process.  The requirements of this distribution $P$ will become clear in
what  follows and  the worked  examples of  Sections~\ref{subsec:discr:action},
\ref{subsec:median} and
\ref{subsec:seq:des}.  We let $\norm{\cdot}$ denote the $L^2(P)$ norm.\\

Let $\Pi$  be a  class of policies,  {\em i.e.}, a  set of  mappings from
  $\xX$ to $\xA$, and let $\Pibar$  be its $L^2(P)$ closure.  We request that
$\Pi$ is not too large, in the sense that
\begin{equation}
  \Pi\textnormal{     is     totally     bounded     with     respect     to
    $\norm{\cdot}$.} \label{eq:Pitb} 
\end{equation}

The value of a policy $\pi\in\Pibar$  is quantified via $\xV(\pi)$.  As in our
earlier result, the regret is defined as $\Reg(\pi)\equiv \xV^\star-\xV(\pi)$,
where $\xV^\star\equiv \sup_{\pi\in\Pi}\xV(\pi)$.

We also introduce the condition
\begin{equation}
  \xV(\cdot)\textnormal{  is uniformly  continuous  on $\Pibar$  with
    respect to $\norm{\cdot}$.} \label{eq:valunifcons} 
\end{equation}

We        assume        that         there        exists        a        class
$\{\IF_\pi  : \pi\in\Pibar\}  \subset L^{2}  (P)$  of mappings  from $\xO$  to
$[-M,M]$    ($M$     not    relying    on    $\pi$)     and    an    estimator
$\{\widehat{\xV}(\pi) :  \pi \in \Pi\}$ of  $\{\xV(\pi) : \pi \in  \Pi\}$ such
that, for a (possibly stochastic) rate $r_n\rightarrow +\infty$,
\begin{gather}
  \label{eq:IFuc}
  \pi\mapsto \IF_{\pi}\textnormal{ is uniformly continuous from
    $\Pibar$ to $L^2(P)$, \quad and} \\
  \label{eq:estseq}
  \sup_{\pi\in\Pi}     \left|r_n\left(\widehat{\xV}(\pi)-\xV(\pi)\right)     -
    \widetilde{\xG}_n \IF_{\pi}\right|=o_P(1).
\end{gather}
In  \eqref{eq:IFuc},  both  spaces   are  equipped  with  $\norm{\cdot}$.   In
\eqref{eq:estseq},  $\widetilde{\xG}_n\in\ell^{\infty}(\xF)$  is a  stochastic
process on $\xF \equiv  \{\IF_\pi : \pi\in\Pi\}$ that may or  may not be equal
to the empirical process, but must satisfy
\begin{equation}
  \widetilde{\xG}_n\leadsto                 \widetilde{\xG}_P\textnormal{                 in
  }\ell^{\infty}(\xF),\label{eq:unifconv} 
\end{equation}
where almost all sample paths  of $\widetilde{\xG}_P$ are uniformly continuous
with respect to $\norm{\cdot}$.

Finally we also assume that, for  the same rate $r_n$ as in (\ref{eq:estseq}),
$\pihat\in \Pi$ satisfies the ERM property
\begin{equation}
  \widehat{\xV}(\pihat)\le  \inf_{\pi\in\Pi}   \widehat{\xV}(\pi)  +
  o_P(r_n^{-1}). \label{eq:generalERM} 
\end{equation}
Like \eqref{eq:valmax}, \eqref{eq:generalERM} is a requirement on the behavior
of  the  optimization  algorithm  on   the  realized  sample,  rather  than  a
statistical  or probabilistic  condition.   The  following result  generalizes
Theorem~\ref{thm:regDonsker}.

\begin{theorem}[More General ERM Result]\label{thm:generalRegDonsker}
  If    (\ref{eq:Pitb})     through    (\ref{eq:generalERM})     hold,    then
  $\Reg(\pihat)=o_P(r_n^{-1})$.
\end{theorem}
A  sketch of  the  \hyperlink{proof:thmgeneralRegDonsker}{proof}  is given  in
Appendix~\ref{app:generalERM}.  We only outline  where the proof would deviate
from that of Theorem~\ref{thm:regDonsker}.

\begin{remark}
  If  $\widetilde{\xG}_P$ is  equal  to the  zero-mean  Gaussian process  with
  covariance  given   by  $\E[\xG_P  f  \xG_P   g]  =  Pfg  -   Pf  Pg$,  then
  $\widetilde{\xG}_P$ has almost surely uniformly continuous sample paths.  In
  particular,   \cite[][Example  1.5.10]{vanderVaartWellner1996}   shows  that
  $\widetilde{\xG}_P$ has almost surely uniformly continuous sample paths with
  respect to the  standard deviation semimetric, and Section 2.1  in that same
  reference  shows  that, for  bounded  $\mathcal{F}$,  one can  replace  this
  semimetric by that on $L^2(P)$.
\end{remark}

In the remainder of  this section, we give a taste  of the broad applicability
of this result via three examples.  The first example is a simple extension of
the framework of  Section~\ref{sec:mainresults} that allows for  more than two
possible actions.   The second example  substitutes the median reward  for the
mean  reward~\citep[for a  similar setting,  see][]{Linnetal2016}.  The  third
example  focuses  again  on  the  mean reward,  but  considers  a  non-i.i.d.,
contextual-bandit-type  setting  in  which  context-specific  actions  may  be
informed by earlier observations \citep{Chambazetal2017}.

\subsection{Example~1: Maximizing the Mean Reward of a Discrete Action} 
\label{subsec:discr:action}

In  this example,  there  are  $\card(\xA) \in  [2,  \infty)$ (finitely  many)
candidate  actions  to   undertake  (think  of  a  discretized   dose  in  the
personalized medicine framework).  Without loss  of generality, we assume that
$\mathcal{A}\subset  [0,1]$.  We  assume that  under the  distribution $P$  of
$O\equiv (X,A,Y)$, the  reward $Y$ is uniformly bounded and  there exists some
$\delta>0$ so  that $\min_{a\in\xA}P(A=a|X)\geq \delta$ with  probability one.
Here   too,   the    value   of   a   policy   $\pi\in    \Pibar$   is   given
by~\eqref{eq:def:value},   the  (within   class   $\Pi$)   optimal  value   is
$\xV^{\star}  \equiv   \sup_{\pi  \in  \Pi}   \xV(\pi)$  and  the   regret  of
$\pi \in \Pibar$ is defined  as in~\eqref{eq:def:regret}.  Finally, we observe
$O_{1}  \equiv  (X_{1},  A_{1},  Y_{1}), \ldots,  O_{n}\equiv  (X_{n},  A_{n},
Y_{n})$ drawn i.i.d. from $P$.

For  each $\pi  \in  \Pibar$ and  $o  \in  \xO$, let  $f_{\pi}  (o)$ be  given
by~\eqref{eq:IF:pi}.                                                       Let
$\widetilde{\xG}_{n} \equiv  n^{1/2} (P_{n} -  P) \in \ell^{\infty}  (\xF)$ be
the  empirical process  on $\xF  =  \{\IF_{\pi} :  \pi \in  \Pi\}$.  Note  the
parallel between the LHS of  \eqref{eq:estseq} with $r_{n} \equiv n^{1/2}$ and
$n^{1/2} \Rem_{n}$ from~\eqref{eq:valest}.

We prove the next lemma in Section~\ref{app:proof:discr:action}:
\begin{lemma}
  \label{lem:discr:action}
  If   $\Pi$  is   Donsker,   then  \eqref{eq:Pitb},   \eqref{eq:valunifcons},
  \eqref{eq:IFuc} and \eqref{eq:unifconv} are met.
\end{lemma}
Therefore, Theorem~\ref{thm:generalRegDonsker} yields the following corollary:
\begin{corollary}
  In the  context of  Section~\ref{subsec:discr:action}, suppose  that $\pihat$
  satisfies \eqref{eq:generalERM} with $r_{n}\equiv n^{1/2}$ and that
  \begin{equation*}
    \sup_{\pi\in\Pi} \left|n^{1/2}\left(\widehat{\xV}(\pi)-\xV(\pi)\right)
      - \widetilde{\xG}_n \IF_{\pi}\right|=o_P(1). 
  \end{equation*}
  If $\Pi$ is Donsker, then $\Reg(\pihat) = o_{P} (n^{-1/2})$.
\end{corollary}

\subsection{Example~2: Maximizing the Median Reward of a Binary Action}
\label{subsec:median}

In this example, we use the same  i.i.d. (from $P$) observed data structure as
in Section~\ref{sec:formalization}, including the strong positivity assumption and the bounds on $A$ and $Y$.    For   every   policy   $\pi   \in   \xP$,   define
$F_\pi              :\xR\rightarrow\xR$              pointwise              by
$F_\pi(m)\equiv  \E\left[P(Y\le  m|A=\pi(X),X)\right]$.    Under  some  causal
assumptions, $F_{\pi}$ is  the cumulative distribution function  of the reward
in the  counterfactual world where  action $\pi(x)$  is taken in  each context
$x \in \xX$.

We define the value  of $\pi$ as the median rather than  the mean reward, {\em
  i.e.}, as
\begin{equation}
  \label{eq:def:median}
  \xV(\pi)\equiv       \inf\left\{m\in\xR        :       1/2\le
    F_{\pi}(m)\right\}. 
\end{equation}
Let $\Pi \subset  \xP$ be the class of candidate  policies. Recall that the regret (within
class    $\Pi$)    of    $\pi    \in     \Pi$    takes    the    form
$\Reg(\pi)\equiv    \sup_{\pi\in\Pi}\xV(\pi)-\xV(\pi)$.

Let us now turn to \eqref{eq:valunifcons}. Assume that there exists $c>0$ such
that, for each $\pi \in \Pi$,  $F_{\pi}$ is continuously differentiable in the
neighborhood $[\xV(\pi)\pm  c]$ with  derivative $\dot{F}_{\pi}(m)$ at  $m$ in
this neighborhood, where
\begin{equation}
  0<\inf_{\pi\in\Pibar}\inf_{m\in        [\xV(\pi)\pm        c]}
  \dot{F}_{\pi}(m)\le                    \sup_{\pi\in\Pibar}\sup_{m\in
    [\xV(\pi)\pm c]} \dot{F}_{\pi}(m)<\infty. \label{eq:derivbd} 
\end{equation}
In     addition,      assume     that     there     exists      a     function
$\omega           :            [0,\infty)\rightarrow[0,\infty)$           with
$\lim_{m\downarrow 0}\omega(m)=\omega(0)=0$ such that
\begin{equation}
  \sup_{\pi\in\Pibar}\sup_{|k| \leq c}\left[\left|\dot{F}_{\pi}(\xV(\pi)+k)                            -
      \dot{F}_{\pi}(\xV(\pi))\right|     -    \omega(|k|)\right]     \le
  0. \label{eq:contderiv} 
\end{equation}

\begin{remark}
  A sufficient, but not necessary, condition  for such an $\omega$ to exist is
  that $F_{\pi}$ is twice continuously  differentiable with the absolute range
  of the second derivative $\ddot{F}_{\pi}$  on $[\xV(\pi)\pm c]$ bounded away
  from infinity  uniformly in  $\pi\in\Pibar$.  Indeed, Taylor's  theorem then
  shows that one can take
  \begin{equation*}
    \omega(m)\equiv    m\times   \sup_{\pi\in\Pibar}\sup_{\tilde{m}\in
      [\xV(\pi)\pm c]}|\ddot{F}_{\pi}(\tilde{m})|. 
\end{equation*}
\end{remark}

The next lemma gives conditions  under which (\ref{eq:valunifcons}) holds. Its
\hyperlink{proof:lemmedianreguc}{proof}          is          given          in
Appendix~\ref{app:proof:median}.

\begin{lemma}\label{lem:medianreguc}
  If    (\ref{eq:derivbd})   and    \eqref{eq:contderiv}    are   met,    then
  (\ref{eq:valunifcons}) holds.
\end{lemma}

For every  $\pi \in \xP$,  $\xV(\pi)$ defined in \eqref{eq:def:median}  can be
viewed as the evaluation at $P$ of the functional
\begin{equation*}
  P'    \mapsto     \inf\big\{m\in\xR    :     1/2\le    \E_{P'}\left[P'(Y\le
      m|A=\pi(X),X)\right]\big\}
\end{equation*}
from  the  nonparametric  model  of distributions  $P'$  satisfying  the  same
constraints  as   $P$  to  the   real  line.   This  functional   is  pathwise
differentiable at $P$ relative to the  maximal tangent space with an efficient
influence function $f_{\pi}$ given by
\begin{multline}
  \IF_{\pi}(o)\equiv
  \frac{\Ind\{a=\pi(x)\}}{P(A=a|X=x)\dot{F}_{\pi}(\xV(\pi))}\left[\Ind\{y\le
    \xV(\pi)\}-P\{Y\le \xV(\pi)|A=a,X=x\}\right] \\
  +             \frac{P\{Y\le            \xV(\pi)|A=\pi(x),X=x\}             -
    1/2}{\dot{F}_{\pi}(\xV(\pi))}.
\end{multline}
Observe that above $\dot{F}_{\pi}  (\xV(\pi))$ only enters $\IF_{\pi}$ as
  a  multiplicative  constant,  and  therefore  an  estimating  equation-based
  estimator or TMLE for $\xV(\pi)$ can be asymptotically linear for $\xV(\pi)$
  without  estimating $\dot{F}_{\pi}$.   We, in  particular, suppose  that we
have  an   estimator  satisfying  (\ref{eq:estseq})  with   $r_n=n^{1/2}$  and
$\widetilde{\mathbb{G}}_n\equiv n^{1/2}(P_n-P)  \in \ell^{\infty}  (\xF)$ with
$\xF \equiv \{\IF_{\pi}  : \pi \in \Pi\}$, see  Remark~\ref{rem:median} at the
end of the present section.

With this choice of $\widetilde{\mathbb{G}}_n$, \eqref{eq:unifconv} is met and
$\Pi$   Donsker   yields   \eqref{eq:Pitb},   as  seen   in   the   proof   of
Theorem~\ref{thm:regDonsker} (the  same argument applies because  the range of
each $\pi\in\xP$ is  bounded in $[-1,1]$). Theorem~\ref{thm:generalRegDonsker}
yields the following corollary:

\begin{corollary}
  In    the    context    of   Section~\ref{subsec:median},    suppose    that
  (\ref{eq:derivbd}) and  \eqref{eq:contderiv} are  met. Let  $\pihat$ satisfy
  \eqref{eq:generalERM} with $r_{n}\equiv n^{1/2}$ and
  \begin{equation*}
    \sup_{\pi\in\Pi} \left|n^{1/2}\left(\widehat{\xV}(\pi)-\xV(\pi)\right)
      - \widetilde{\xG}_n \IF_{\pi}\right|=o_P(1). 
  \end{equation*}
  If     $\Pi$    is     Donsker    and     \eqref{eq:IFuc}    holds,     then
  $\Reg(\pihat) = o_{P} (n^{-1/2})$.
\end{corollary}

Since  $\Pi$  is Donsker,  \eqref{eq:IFuc}  can  be derived  under  regularity
conditions  by using  essentially  the  same techniques  as  in  the proof  of
Theorem~\ref{thm:regDonsker}.        A       slight      modification       to
Lemma~\ref{lem:Pi:Donsker:F:Donsker} is needed.  The main regularity condition
consists      in     assuming      that      the     real-valued      function
$\pi\mapsto \dot{F}_\pi(\xV(\pi))$ over  $\Pibar$ equipped with $\norm{\cdot}$
is    uniformly     continuous,    see     Corollary~\ref{cor:derivcont}    in
Appendix~\ref{app:generalERM}.   It  is  tedious,  though  not  difficult,  to
complement this main regularity condition with secondary conditions.  It would
suffice   to   restrict   the   (uniform    in   $\pi$)   behavior   for   all
$P\{Y\le v|A=\pi(x),X=x\}$ across all real $v$.

\begin{remark}
  \label{rem:median}
  One can, for  example, establish conditions under  which the estimating
    equation-based estimator defined as a solution in $v$ to
    \begin{equation*}
      \frac{\Ind\{a=\pi(x)\}\left[\Ind\{Y\le
          v\}-\widehat{P}\{Y\le v|A=\pi(X),X\}\right]}{\widehat{P}(A=a|X=x)} +
      \widehat{P}\{Y\le v|A=\pi(X),X\} = 1/2,
    \end{equation*}
    satisfies (\ref{eq:estseq}).  In the above display, $\widehat{P}$ denotes
  an estimate of (certain conditional  probabilities under) $P$. One could use
  cross-validated estimating  equations to  avoid the  need for  any empirical
  process conditions on $\widehat{P}$.
\end{remark}

\subsection{Example~3: Sequential Decisions  to Maximize the Mean  Reward of a
  Binary Action}
\label{subsec:seq:des}

\noindent{\bf     The    sampling     design.}      This    section     builds
upon~\cite{Chambazetal2017}.  Let $\{(X_{n},  Z_{n}(-1), Z_{n}(1)\}_{n\geq 1}$
be  a sequence  of random  variables drawn  independently from  a distribution
$\bP$  such  that,  if  $\{\tilde{A}_{n}\}_{n   \geq  1}$  is  a  sequence  of
independent actions drawn  from the conditional distribution of $A$ (an action with support $\{-1,1\}$) given $X$
under the same  $P$ as in Section~\ref{sec:formalization},  then the resulting
sequence $\{(X_{n}, \tilde{A}_{n}, Z_{n}  (\tilde{A}_{n}))\}_{n \geq 1}$ is an
i.i.d. sample  from $P$.  For  notational simplicity,  we assume that  all the
random variables $Z_{n} (-1)$ and $Z_{n} (1)$ take their values in $[0,1]$.

In    this    example,    however,    the    distribution    $\nu_{n}$    of
$(O_{1}, \ldots,  O_{n})$ does not  write as  a product because  the actions
$A_{1}, \ldots,  A_{n}$ are  not i.i.d. On  the contrary,  once sufficiently
many  observations   have  been  accrued   to  carry  out   inference  then,
sequentially, each new  randomized action $A_{n}$ is  drawn conditionally on
$X_{n}$   and   an  estimate   derived   from   the  previous   observations
$O_{1}  \equiv  (X_{1},  A_{1},   Y_{1}),  \ldots,  O_{n-1}\equiv  (X_{n-1},
A_{n-1},   Y_{n-1}    )$   yielded   by   the    previous   actions,   where
$Y_{i} \equiv  Z_{i} (A_{i})$ for $i=1,  \ldots, n-1$.  Let us  describe how
the data-adaptive design unfolds.

Let $\{t_{n}\}_{n \geq  1}$ and $\{\xi_{n}\}_{n \geq 1}$  be two nonincreasing
sequences   with    $t_{1}   \leq   1/2$,    $\lim_{n}   t_{n}   >    0$   and
$\lim_{n} \xi_{n} > 0$.   For each $n \geq 1$, let  $G_{n}$ be a nondecreasing
$\kappa_{n}$-Lipschitz functions approximating $u  \mapsto \Ind\{u\geq 0\}$ in
the sense that $G_{n}  (u) = t_{n}$ for $u \leq -\xi_{n}$,  $G_{n} (0) = 1/2$,
$G_{n}  (u)   =  1-t_{n}$  for  $u   \geq  \xi_{n}$.   We  also   suppose  that
$\limsup_{n}   \kappa_{n}    <   \infty$.     For   each   $n\geq    1$,   let
$\xQ_{n} \equiv \{Q_{\beta} : \beta \in B_{n}\}$ be a working model consisting
of   functions    mapping   $\{-1,1\}   \times   \xX$    to   $[0,1]$.    Each
$Q_{\beta} \in  \xQ_{n}$ can  be viewed  as a  candidate approximation  to the
function
$(A,X) \mapsto \Ind\{A=1\}\E_{\bP}[Z(1)|X] + \Ind\{A=-1\}\E_{\bP}[Z(-1)|X]$.

For some $n_{0} \geq 1$, $A_{1}, \ldots, A_{n_{0}}$ are independent draws from
the  conditional  distribution  of  $A$   given  $X$  under  $P$.   For  every
$i=1, \ldots,  n_{0}$, let  $g_{i}:\{-1,1\} \times \xX  \to [0,1]$  be defined
pointwise by $g_{i} (a,x) \equiv P(A = a|X=x)$.  Set $n > n_{0}$, suppose that
we have fully specified how $O_{1}, \ldots, O_{n-1}$ have been sampled through
the description of how $A_{1}, \ldots,  A_{n-1}$ have been randomly drawn from
Rademacher laws\footnote{The  Rademacher law with  parameter $p \in  [0,1]$ is
  the law  of $A  \in \{-1,1\}$  such that  $A=1$ with  probability~$p$.} with
parameters $g_{1}(1,X_{1}),  \ldots, g_{n-1}(1,X_{n-1})$,  respectively.  Now,
let us  describe how $O_{n}$  is sampled by  specifying how action  $A_{n}$ is
randomized,  therefore  completing the  description  of  the sampling  design.
Based on $O_{1}, \ldots, O_{n-1}$, we define
\begin{equation*}
  \beta_{n}  \equiv   \mathop{\arg\min}_{\beta  \in   B_{n}}  \sum_{i=1}^{n-1}
  \frac{\xL(Q_{\beta})(O_{i}) }{g_{i} (A_{i} | X_{i})} 
\end{equation*}
where    the    least-square   loss    function    $\xL$    is   given    by
$\xL(Q_{\beta})  (O)\equiv  (Y  -  Q_{\beta}(A,X))^{2}$.   This  yields  the
mapping $g_{n}  : \{-1,1\}  \times \xX \to  [t_{n}, 1-t_{n}]  \subset [0,1]$
given by
\begin{equation*}
  1 - g_{n}(-1,x) \equiv g_{n}(1,x)  \equiv G_{n} \big(Q_{\beta_{n}} (1,x) -
  Q_{\beta_{n}} (-1,x)\big).  
\end{equation*}
Once   $X_{n}$    is   observed,   we   sample    $A_{n}$   conditionally   on
$O_{1},  \ldots,  O_{n-1}$ and  $X_{n}$  from  the Ra\-de\-ma\-cher  law  with
parameter  $g_{n}(1,X_{n})$,  we  then   carry  out  action  $A_{n}$,  observe
$Y_{n} \equiv Z_{n} (A_{n})$, and form $O_{n} \equiv (X_{n}, A_{n}, Y_{n})$.

\begin{remark}
The  sampling design  looks  for a  trade-off  between exploration  and
  exploitation.   Here,  exploitation consists  in  using  the current  best
  estimates  $Q_{\beta_{n}}(1,  X_{n})$  and $Q_{\beta_{n}}(-1,  X_{n})$  of
  $\E_{\bP}[Z(1)|X=X_{n}]$ and $\E_{\bP}[Z(-1)|X=X_{n}]$ to favor the action
  which seems  to have the  larger mean reward  in context $X_{n}$.   If the
  absolute value of the difference between  the two estimates is larger than
  $\xi_{n}$,  then  the  supposedly  superior action  is  carried  out  with
  probability  $(1-t_{n})$. Exploration  consists in  giving the  supposedly
  inferior action a probability $t_{n}$  to be undertaken nonetheless.  This
  allows,  for   instance,  to  correct   a  possibly  poor   estimation  of
  $\E_{\bP}[Z(1)|X]$ and $\E_{\bP}[Z(-1)|X]$ in some strata of $\xX$.
\end{remark}

\begin{remark}
  For  any  $g  :  \{-1,1\}  \times  \xX \to  [0,1]$,  let  $\bP^{g}$  be  the
  distribution of  $O$ defined by  {\em (i)}~sampling $(X, Z(-1),  Z(1))$ from
  $\bP$,  {\em (ii)}  sampling  $A$  from the  Rademacher  law with  parameter
  $g(1,X)$,  {\em  (iii)}  setting  $O   \equiv  (X,  A,  Z(A))$.   Note  that
  $W \equiv g(A,X)$ is a deterministic function of $g$ and $O$.  Therefore, we
  can augment  $O$ with  $W$, {\em  i.e.}, substitute  $(O,W)$ for  $O$, while
  still  denoting $(O,W)  \sim  \bP^{g}$.  A  balanced  design corresponds  to
  $\bP^{b}$, where  $b : \{-1,1\} \times  \xX \to [0,1]$ only  takes the value
  1/2.
\end{remark}

\noindent{\bf  Optimal policy  estimation.}  We  now define  set of  candidate
policies $\Pi$.   In particular $\Pi$  is the set  of policies such  that, for
each $\pi \in \Pi$, there exists $\beta \in \cup_{n\geq 1} B_{n}$ for which
\begin{eqnarray*}
  \pi (X) 
  &\equiv& \mathop{\arg\max}_{a \in \{-1,1\}} Q_{\beta} (a, X)\\
  &=& \Ind\{Q_{\beta} (1,X) \geq  Q_{\beta} (-1,X)\} - \Ind\{Q_{\beta} (1,X)
      < Q_{\beta} (-1,X)\},
\end{eqnarray*}
where the  choice to  let $A=1$  when $Q_{\beta}(1,X)=Q_{\beta}(-1,X)$  in the
second equality  is a  convention.  The  value $\xV(\pi)$  of $\pi\in  \Pi$ is
still  given  by \eqref{eq:def:value}  and  its  (within class  $\Pi$)  regret
$\Reg(\pi)$ by \eqref{eq:def:regret}.

Recall   the  definition   of   $\xF$   characterized  by   \eqref{eq:IF:pi}
and~\eqref{eq:IF}.   Let $\widetilde{\xG}_{n}  \in  \ell^{\infty} (\xF)$  be
such that, for each $\IF \in \xF$,
\begin{equation*}
  n^{1/2}  \widetilde{\xG}_{n}  \IF   \equiv  \sum_{i=1}^{n}  \left(\IF(O_{i},
    W_{i})   -   \E_{\bP^{g_{i}}}   [\IF(O_{i},  W_{i})   |   O_{1},   \ldots,
    O_{i-1}]\right), 
\end{equation*}
where we  use the fact  that $g_{i}$ is  deterministic when one  conditions on
$O_{1}, \ldots, O_{i-1}$  ({\em i.e.}, on nothing when  $i=1$).  By exploiting
the  martingale structure  of each  $n^{1/2} \widetilde{\xG}_{n}  \IF$, it  is
possible  to  show  that $\widetilde{\xG}_{n}$  satisfies  \eqref{eq:unifconv}
under mild conditions (expressed in terms  of the uniform entropy integral) on
$\{\xQ_{n}\}_{n        \geq       1}$~\cite[][Lemma~B.3]{Chambazetal2017supp}.
Theorem~\ref{thm:generalRegDonsker}  yields  a  corollary for  our  sequential
design  setting. We  discuss the  conditions  of the  corollary following  its
statement.
\begin{corollary}
  \label{cor:seq:des}
  In   the    context   of   Section~\ref{subsec:seq:des},    suppose   that
  $\{\xQ_{n}\}_{n \geq 1}$  is chosen in such a way  that $\Pi$ is separable
  and  $J_{1} (\eta_{n},  \Pi)  = o(1)$  for any  $\eta_{n}  = o(1)$,  where
  $J_{1} (\eta, \Pi)$ is the uniform entropy integral evaluated at $\eta$ of
  $\Pi$ wrt the envelope function constantly equal to one. Moreover, suppose
  that $\pihat$  satisfies \eqref{eq:generalERM} with  $r_{n}\equiv n^{1/2}$
  and that
  \begin{equation*}
    \sup_{\pi\in\Pi} \left|n^{1/2}\left(\widehat{\xV}(\pi)-\xV(\pi)\right)
      - \widetilde{\xG}_n \IF_{\pi}\right|=o_P(1). 
  \end{equation*}
  Then $\Reg(\pihat) = o_{P} (n^{-1/2})$.
\end{corollary}

Lemma~\ref{lem:discr:action}    shows    that    \eqref{eq:valunifcons}    and
\eqref{eq:IFuc} are  met.  Conditions \eqref{eq:Pitb}  and \eqref{eq:unifconv}
follow          from          Lemma~\ref{lem:Pi:Donsker:F:Donsker}          in
Appendix~\ref{sec:proofregDonsker}  and  from  the maximal  inequality  stated
in~\cite[][Lemma~B.3]{Chambazetal2017supp},   whose   assumptions  drive   the
conditions on $\{\xQ_{n}\}_{n \geq  1}$ stated in Corollary~\ref{cor:seq:des}.
A    concrete    example    of    such     a    sequence    can    be    found
in~\cite[][Section~4.4]{Chambazetal2017} (see  also Section~4.1 therein  for a
brief  reminder   about  the  uniform  integral   entropy).   Additional  mild
assumptions  guaranteeing  that  an  estimator  $\xVhat$  can  be  defined  as
requested      by     Corollary~\ref{cor:seq:des}      can     be      adapted
from~\cite{Chambazetal2017}.    Finally,  we   note  for   the  ERM   property
\eqref{eq:generalERM}  that  each $\pi\in\Pi$  is  known  once one  knows  the
$\beta\in\cup_{n\ge 1}  B_n$ that indexes  this $\pi$.  Hence, it  suffices to
study  the estimated  value  $\widehat{\xV}(\pi)$ for  each  such $\pi$.   The
appropriate algorithm  for carrying out  this optimization will depend  on the
particular choice of classes $B_n$ and the working model $Q_{\beta}$.

\section{Discussion}

We have  presented fast rates  of regret  decay in optimal  policy estimation,
{\em i.e.},  rates of  decay that  are faster than  the rate  of decay  of the
standard error of an  efficient estimator of the value of  any given policy in
the  candidate   class.   Our  method   of  proof  for  our   primary  result,
Theorem~\ref{thm:regDonsker}, leverages  the fact  that the  empirical process
over a Donsker class converges in distribution to a Gaussian process, and that
the  sample paths  of  this  limiting process  are  (almost surely)  uniformly
continuous. The downside of our analysis, namely passing to the limit and then
studying the behavior of  the limiting process, is that it  does not appear to
allow one to obtain a faster  than $o_P(n^{-1/2})$ rate of convergence for the
regret.

It would  be of  interest to  replace our limiting  argument by  finite sample
results that would  allow one to exploit the finite  sample equicontinuity of
the empirical process  to demonstrate faster rates. Nonetheless,  we note that
the problem-dependent margin  condition will often have a major  impact on the
extent to  which the  rate can  be improved. Furthermore,  as we  discussed in
Section~\ref{rem:tightening},  the existence  of the  remainder term  $\Rem_n$
necessitates a careful  consideration of whether or not  the empirical process
term represents the dominant error term, or whether the second-order remainder
that appears  due to the non-linearity  of the value parameter  represents the
dominant error  term. In restricted,  semiparametric models, we  suspect that,
depending on the underlying margin, an  inefficient estimator of the value may
yield a faster rate of regret  decay than an efficient estimator. Despite this
surprising  phenomenon,  we  continue  to  advocate  the  use  of  first-order
efficient value estimators.

{\singlespacing
  \section*{Acknowledgements}
  Alex Luedtke gratefully acknowledges the support of the New Development Fund
  of the Fred Hutchinson Cancer  Research Center. Antoine Chambaz acknowledges
  the support  of the  French Agence  Nationale de  la Recherche  (ANR), under
  grant  ANR-13-BS01-0005 (project  SPADRO) and  that this  research has  been
  conducted as part of the project Labex MME-DII (ANR11-LBX-0023-01).  }

\newpage
\appendix
\setcounter{equation}{0}
\renewcommand{\theequation}{\thesection.\arabic{equation}}
\setcounter{theorem}{0}
\renewcommand{\thetheorem}{\thesection.\arabic{theorem}}
\renewcommand{\thecorollary}{\thesection.\arabic{theorem}}
\renewcommand{\thelemma}{\thesection.\arabic{theorem}}

\section*{Appendix}

\section{Proof of Main ERM Result}\label{sec:proofregDonsker}

We begin  this section  with five  lemmas and a  corollary, whose  proofs only
require    results    from    functional     analysis.     We    then    prove
Lemma~\ref{lem:Pi:Donsker:F:Donsker}     and     Theorem~\ref{thm:regDonsker},
using results from empirical process theory.

In   Lemmas~\ref{lem:closurecompact}~through~\ref{lem:cont}  to   follow,  all
topological  results make  use  of  the strong  topology  on  $L^2(P)$. A  set
$S\subset L^2(P)$ is called totally  bounded if, for every $\epsilon>0$, there
exists  a  finite  collection  of radius-$\epsilon$  open  balls  that  covers
$S$. When  we refer to  convergence in these  lemmas, we refer  to convergence
with respect  to $\norm{\cdot}$. A  set $S$  is sequentially compact  if every
sequence  of   elements  of   $S$  has  a   convergent  subsequence.    A  set
$S\subset   L^2(P)$  is   compact   if   every  open   cover   has  a   finite
subcover. Sequential compactness and compactness are equivalent for subsets of
metric    spaces.      Lemma~\ref{lem:gsemicont}    and    the     proof    of
Theorem~\ref{thm:regDonsker}  use several  additional  definitions, which  are
given immediately before Lemma~\ref{lem:gsemicont}.

\begin{lemma}\label{lem:closurecompact}
  If $\Pi$ is totally bounded, then $\Pibar$ is compact.
\end{lemma}
\begin{proof}
  The  Hilbert  space  $L^2(P)$  is   a  complete  metric  space.   Therefore,
  \citep[][Corollary  6.65]{Browder2012}  implies  that  $\Pi$  is  relatively
  compact. In other words, $\Pibar$ is compact.
\end{proof}

For   brevity,   introduce   $Q(A,X)   \equiv   \E(Y|A,X)$.   Note   that
  $\gamma(X) = Q(1,X) - Q(-1,X)$.

\begin{lemma}\label{lem:VRcont}
  Both $\xV$ and $\Reg$ are continuous on $\Pibar$.
\end{lemma}

\begin{proof}
  The  key to  this proof  is the  following simple  remark: for  every policy
  $\pi:\xX \to \{-1,1\}$,
  \begin{equation}
    \label{eq:key:one}
    2Q(\pi(X), X) = (1+\pi(X)) Q(1,X) + (1-\pi(X)) Q(-1,X).
  \end{equation}
  Choose arbitrarily $\pi_{1}, \pi_{2} \in \Pibar$.  By \eqref{eq:key:one} and
  the Cauchy-Schwarz inequality, it holds that
  \begin{equation}
    \label{eq:nupione:nupitwo}
    2|\xV(\pi_{1})      -     \xV(\pi_{2})|      =     \left|\E\left[\gamma(X)
        (\pi_{1}-\pi_{2})(X)\right]\right|    \leq     \norm{\gamma}    \times
    \norm{\pi_{1} - \pi_{2}}.  
  \end{equation}
  This proves the continuity of $\xV$. That of $\Reg$ follows immediately.
\end{proof}

Recall the definition \eqref{eq:Pibar} of $\Pibar^{\star}$.
\begin{lemma}\label{lem:Pistarnonempty}
  If  $\Pi$  is  totally   bounded,  then  $\inf_{\pi\in  \Pibar}\Reg(\pi)=0$.
  Moreover,
  $\Pibar^\star =  \{\pi^{\star}\in \Pibar  : \Reg(\pi^{\star}) \leq  0\}$ and
  $\Pibar^\star\neq \emptyset$.
\end{lemma}

\begin{proof}
  By    Lemma~\ref{lem:closurecompact},     $\Pibar$    is     compact.     By
  Lemma~\ref{lem:VRcont},  $\Reg$  is  continuous.  Thus,  $\Reg$  admits  and
  achieves   a  minimum   $\Reg(\pibar)  =   \inf_{\pi\in\Pibar}\Reg(\pi)$  on
  $\Pibar$.     Since    $\inf_{\pi\in\Pi}\Reg(\pi)=0$,     we    know    that
  $\Reg(\pibar)  \le  0$.  In  fact,   a  contradiction  argument  shows  that
  $\Reg(\pibar)         =         0$.          In         other         words,
  $\Pibar^\star =  \{\pi^{\star}\in \Pibar  : \Reg(\pi^{\star}) \leq  0\}$ and
  $\pibar \in \Pibar^{\star}$, hence $\Pibar^{\star} \neq \emptyset$.

  Indeed,  assume  that  $\Reg(\pibar)<0$.  Because  $\pibar\in\Pibar$,  there
  exists  a sequence  $\{\pi_m\}_{m\geq 1}$  of  elements of  $\Pi$ such  that
  $\norm{\pi_m-\pibar}\rightarrow   0$.    By   continuity  of   $\Reg$   (see
  Lemma~\ref{lem:VRcont}),  $\Reg(\pi_m)\rightarrow   \Reg(\pibar)<0$.   Thus,
  $\inf_{\pi\in\Pi}\Reg(\pi)<0$. Contradiction.
\end{proof}

\begin{lemma}\label{lem:cont}
  If $\Pi$ is totally bounded, then
  \begin{equation}
    \label{eq:ca}\tag{CA}
    \limsup_{r\downarrow    0}     \sup_{\pi\in\Pibar    :     \Reg(\pi)\le    r}
    \inf_{\pi^\star\in \Pibar^\star} \norm{\pi^\star - \pi}=0.  
  \end{equation}
\end{lemma}

\begin{proof}
  We argue by  contraposition. Suppose \ref{eq:ca} does not  hold.  Then there
  exists  a  sequence  $\{\pi_m\}_{m\geq  1}$  of  elements  of  $\Pibar$  and
  $\delta>0$ such that $\Reg(\pi_m)\rightarrow 0$ and, for all $m\geq 1$,
  \begin{equation}
    \label{eq:contradiction}
    \inf_{\pi^\star\in\Pibar^\star}\norm{\pi_m-\pi^\star}>\delta. 
  \end{equation}

  We now give a contradiction argument to show that $\{\pi_m\}_{m\geq 1}$ does
  not  have a  convergent  subsequence.  Suppose  there  exists a  subsequence
  $\{\pi_{m_k}\}_{k\geq 1}$  such that $\norm{\pi_{m_k}- \pi_{\infty}}  \to 0$
  for   some   $\pi_{\infty}\in   L^2(P)$.    Now,   note   that   {\em   (i)}
  $\pi_{\infty}\in    \Pibar$,   the    closure   of    $\Pi$,   {\em    (ii)}
  $\Reg(\pi_{m_k})\rightarrow        0$,         and        {\em        (iii)}
  $\Reg(\pi_{m_k})\rightarrow  \Reg(\pi_\infty)$   by  Lemma~\ref{lem:VRcont}.
  Consequently,  $\Reg(\pi_\infty)=0$.   Since  $\pi_{\infty}\in\Pibar$,  this
  reveals that $\pi_{\infty}\in\Pibar^\star$ and
  \begin{equation*}
    \inf_{\pi^\star\in\Pibar^\star}\norm{\pi_{m_k}-\pi^\star}\le
    \norm{\pi_{n_k}-\pi_{\infty}}\rightarrow 0, 
  \end{equation*}
  in contradiction with \eqref{eq:contradiction}.   Thus, there does not exist
  a convergent subsequence $\{\pi_{m_k}\}_{k\geq 1}$ of $\{\pi_m\}_{m\geq 1}$,
  completing the contradiction argument.

  We now  return to the contraposition  argument. The existence of  a sequence
  $\{\pi_m\}_{m\geq  1}$  of elements  of  $\Pibar$  not having  a  convergent
  subsequence  implies   that  $\Pibar$  is  not   sequentially  compact  and,
  therefore, that it is not compact.  By Lemma~\ref{lem:closurecompact}, $\Pi$
  is not totally bounded. This completes the proof.
\end{proof}

Recall  the  definition  \eqref{eq:IF}  of  $\xF$. The  next  lemma  uses  the
following definitions.  We let $\ell^{\infty}(\xF)$ denote the metric space of
all bounded functions $z : \xF\rightarrow\xR$, equipped with the supremum norm
and,  for   any  $r>0$,  $\Pi_r\equiv\{\pi\in\Pi  :   \Reg(\pi)\le  r\}$.   If
$\Pibar^\star$  is nonempty  (for instance,  if  $\Pi$ is  totally bounded  by
Lemma~\ref{lem:Pistarnonempty}) we let, for any $\pi^\star\in\Pibar^\star$ and
$s>0$, $B_{\Pi}(\pi^\star,s)\equiv \{\pi\in\Pi  : \norm{\pi^\star-\pi}\le s\}$
denote  the  intersection   of  the  radius-$s$  $L^2(P)$   ball  centered  at
$\pi^\star$ and  the collection $\Pi$.  Because  $\Pibar^\star\subset \Pibar$,
$B_{\Pi}(\pi^\star,s)$ is nonempty.

\begin{lemma}\label{lem:gsemicont}
  Define $g : \ell^\infty(\xF)\times(0,\infty)\rightarrow\xR$ as
  \begin{equation*}
    g(z,r)\equiv  \sup_{\pi\in  \Pi_r}  z(\IF_{\pi}) -  \limsup_{s\downarrow  0}
    \sup_{\pi^\star\in\Pibar^\star}      \inf_{\pi\in      B_{\Pi}(\pi^\star,s)}
    z(\IF_{\pi}). 
  \end{equation*}
  Let $z  \in \ell^{\infty} (\xF)$ be  $\norm{\cdot}$-uniformly continuous and
  let   $\{(z_m,r_m)\}_{m\geq    1}$   be   a   sequence    with   values   in
  $\ell^\infty(\xF)\times(0,\infty)$                 such                 that
  $\sup_{f\in\xF}|z_m(f)  - z(f)|  +  |r_{m}|  \to 0$.   If  $\Pi$ is  totally
  bounded, then
  \begin{equation*}
    \limsup_{m\rightarrow\infty} g(z_m,r_m)\le 0.
  \end{equation*}
\end{lemma}

The following corollary will prove useful.
\begin{corollary}
  \label{lem:corollary:gsemicont}
  Recall   the  definition   of  $g$   from  Lemma~\ref{lem:gsemicont}.    Let
  $h :  \ell^\infty(\xF)\times[0,\infty)\rightarrow\xR$ be such that,  for all
  $z \in \ell^\infty(\xF)$,
  \begin{equation*}
    h(z,  r)  \equiv  \max\left(g(z,r),  0\right) \textnormal{  if  $r>0$  and
      $h(z,0) \equiv 0$}.
  \end{equation*}
  Let $z\in \ell^{\infty} (\xF)$ be $\norm{\cdot}$-uniformly continuous.  If
  $\Pi$ is totally bounded, then $h$ is continuous at $(z,0)$.
\end{corollary}

\begin{proof}[Proof          of         Lemma~\ref{lem:gsemicont}          and
  Corollary~\ref{lem:corollary:gsemicont}]
  Fix a  sequence $\{(z_m,r_m)\}_{m\geq 1}$  satisfying the conditions  of the
  Lemma~\ref{lem:gsemicont}. Observe that, for every $m\geq 1$,
  \begin{eqnarray*}
    |g(z_m,r_m) - g(z,r_m)| 
    &\leq& \left|\sup_{\pi\in\Pi_{r_{m}}}
           z_{m}(\IF_{\pi}) - \sup_{\pi\in\Pi_{r_{m}}} z(\IF_{\pi}) \right|\\
    && + \left|\limsup_{s\downarrow 0}\sup_{\pi^\star\in\Pibar^\star}\inf_{\pi\in
       B_{\Pi}(\pi^\star,s)}    z_{m}(\IF_{\pi})    -    \limsup_{s\downarrow
       0}\sup_{\pi^\star\in\Pibar^\star}\inf_{\pi\in    B_{\Pi}(\pi^\star,s)}
       z(\IF_{\pi})\right|\\
    &\leq& 2 \sup_{\pi\in\Pi} |z_{m}(\IF_{\pi}) - z(\IF_{\pi})|
  \end{eqnarray*}
  where  the  above  RHS  expression  is  $o(1)$  because  $z_{m}  \to  z$  in
  $\ell^{\infty}(\xF)$. Therefore,
  \begin{eqnarray}
    g(z_m,r_m)
    &=& g(z_m,r_m) - g(z,r_m) + g(z,r_m)  \nonumber \\
    &=& g(z_m,r_m) - g(z,r_m) + 
        \left[\sup_{\pi\in\Pi_{r_m}}  z(\IF_{\pi})   -  \limsup_{s\downarrow
        0}\sup_{\pi^\star\in\Pibar^\star}\inf_{\pi\in        B_{\Pi}(\pi^\star,s)}
        z(\IF_{\pi})\right]. \nonumber\\
    &\leq& \left[\sup_{\pi\in\Pi_{r_m}} z(\IF_{\pi}) - \limsup_{s\downarrow 
           0}\sup_{\pi^\star\in\Pibar^\star}\inf_{\pi\in        B_{\Pi}(\pi^\star,s)}
           z(\IF_{\pi})\right] + o(1).  \label{eq:gcont} 
  \end{eqnarray}
  Let us  show now that the  RHS expression in \eqref{eq:gcont}  is $o(1)$.

  For any $r> 0$, there exists a $\pi_r\in\Pi_r$ such that
  \begin{equation}
    \label{eq:approx:RHS:one}
    \sup_{\pi\in\Pi_r} z(\IF_{\pi}) \leq z(\IF_{\pi_r}) + r.
  \end{equation}
  Furthermore, there exists a $\pi_r^\star\in\Pibar^\star$ such that
  \begin{equation}
    \norm{\pi_r^\star-\pi_r}\le \inf_{\pi^\star\in \Pibar^\star}
    \norm{\pi^\star - \pi_r} + r\le \sup_{\pi\in\Pi_r}
    \inf_{\pi^\star\in \Pibar^\star} \norm{\pi^\star - \pi} +
    r. \label{eq:invokecont} 
  \end{equation}
  Likewise, there exists $\tilde{\pi}_r\in B_{\Pi}(\pi_r^\star,r)$ such that
  \begin{eqnarray}
    \notag
    z(\IF_{\tilde{\pi}_r}) 
    &\le& \inf_{\pi\in B_{\Pi}(\pi_r^\star,r)} z(\IF_{\pi})+r\le 
          \sup_{\pi^\star\in\Pibar^\star}\inf_{\pi\in    B_{\Pi}(\pi^\star,r)}
          z(\IF_{\pi})+r \\ 
    \label{eq:approx:RHS:three}
    &=&  \left(\limsup_{s\downarrow  0}\sup_{\pi^\star\in\Pibar^\star}\inf_{\pi\in
        B_{\Pi}(\pi^\star,s)} z(\IF_{\pi}) +o_{r}(1) \right) + r,
  \end{eqnarray} 
  where the above equality holds by  the definition of the limit superior (the
  $o_{r}(1)$ above  represents the term's  behavior as $r\rightarrow  0$).  In
  light of \eqref{eq:approx:RHS:one}  and \eqref{eq:approx:RHS:three}, we thus
  have
  \begin{equation}
    \sup_{\pi\in\Pi_{r}}      z(\IF_{\pi})       -      \limsup_{s\downarrow
      0}\sup_{\pi^\star\in\Pibar^\star}\inf_{\pi\in      B_{\Pi}(\pi^\star,s)}
    z(\IF_{\pi})  \le  z(\IF_{\pi_{r}}) -  z(\IF_{\tilde{\pi}_{r}})  +
    2r + o_{r}(1).\label{eq:useUC} 
  \end{equation}

  By    the     $\norm{\cdot}$-uniform         continuity         of         $z$,
  $|z(\IF_{\pi_{r}})   -    z(\IF_{\tilde{\pi}_{r}})|   =   o_{r}    (1)$   if
  $\norm{\IF_{\pi_{r}}  - \IF_{\tilde{\pi}_{r}}}  = o_{r}  (1)$.  Let  us show
  that       the      latter       condition       is      met.        Because
  $\tilde{\pi}_r\in   B_{\Pi}(\pi_r^\star,r)$,    the   triangle   inequality,
  \eqref{eq:invokecont} and \eqref{eq:fpione:fpitwo} imply that
  \begin{eqnarray}
    \notag
    \norm{\IF_{\pi_r} - \IF_{\tilde{\pi}_r}} 
    &\le&   \norm{\IF_{\pi_r}   -
          \IF_{\pi_r^\star}} + \norm{\IF_{\pi_r^\star}-\IF_{\tilde{\pi}_r}}\\
    &\lesssim& \norm{\pi_r -  \pi_r^\star} + \norm{\pi_{r}^{\star} -
               \tilde{\pi}_{r}}  \leq   \sup_{\pi\in\Pi_r}  \inf_{\pi^\star\in
               \Pibar^\star}      \norm{\pi^\star     -      \pi}     +      2
               r. \label{eq:IFscollapsing} 
  \end{eqnarray}
  Because     $\Pi_r\subset     \{\pi\in\Pibar    :     \Reg(\pi)\le     r\}$,
  Lemma~\ref{lem:cont}  (which  applies  because  $\Pi$  is  totally  bounded)
  implies that
  \begin{equation*}
    \limsup_{r\downarrow     0}      \sup_{\pi\in\Pi_r}     \inf_{\pi^\star\in
      \Pibar^\star} \norm{\pi-\pi^\star}=0,
  \end{equation*}
  from which we deduce that the RHS of (\ref{eq:IFscollapsing}) is $o_{r}(1)$.
  In  summary, $\norm{\IF_{\pi_{r}}  -  \IF_{\tilde{\pi}_{r}}}  = o_{r}  (1)$,
  hence  $|z(\IF_{\pi_{r}})  -  z(\IF_{\tilde{\pi}_{r}})|  =  o_{r}  (1)$  and
  consequently, by \eqref{eq:useUC}, 
  \begin{equation*}
    \sup_{\pi\in\Pi_{r}}      z(\IF_{\pi})       -      \limsup_{s\downarrow
      0}\sup_{\pi^\star\in\Pibar^\star}\inf_{\pi\in      B_{\Pi}(\pi^\star,s)}
    z(\IF_{\pi}) = o_{r}(1).
  \end{equation*}

  Taking  $r  =  r_{m}$ and  using  the  above  result  reveals that  the  RHS
  expression in \eqref{eq:gcont} is $o(1)$ indeed. This completes the proof of
  Lemma~\ref{lem:gsemicont}.

  In    the    context   of    Corollary~\ref{lem:corollary:gsemicont},    let
  $\{(z_{m},   r_{m})\}_{m   \geq  1}$   be   a   sequence  with   values   in
  $\ell^{\infty}(\xF)      \times      [0,\infty)$     and      such      that
  $\sup_{f\in\xF}|z_{m}(f)     -      z(f)|     +     |r_{m}|      \to     0$.
  Lemma~\ref{lem:gsemicont}                    implies                    that
  $h(z_{m}, r_{m})  \to h(z,0) \equiv  0$. Since the sequence  was arbitrarily
  chosen,   $h$  is   indeed  continuous   at   $(z,0)$  and   the  proof   of
  Corollary~\ref{lem:corollary:gsemicont} is complete.
\end{proof}

\begin{lemma}
  \label{lem:Pi:Donsker:F:Donsker}
  If $\Pi$ is Donsker, then $\xF$ is also Donsker.
\end{lemma}

  \begin{proof}
    Similar to \eqref{eq:key:one}, the key  is the following remark: for every
    policy $\pi : \xX \to \{-1,1\}$,
    \begin{eqnarray}
      \notag
      2 \IF_{\pi} (O) 
      &=& |A+\pi(X)| \frac{Y - Q(A,X)}{P(A|X)} \\
      \label{eq:key:two}
      && + (1+\pi(X)) Q(1,X) + (1-\pi(X)) Q(-1,X) - 2 \xV(\pi). 
    \end{eqnarray}
    For   future   use,   we    first   note   that   \eqref{eq:key:one}   and
    \eqref{eq:nupione:nupitwo} imply, for any $\pi_{1}, \pi_{2} \in \Pi$,
    \begin{equation*}
      |\IF_{\pi_{1}}(O) -  \IF_{\pi_{2}}(O)| \lesssim  |\pi_{1} (X)  - \pi_{2}
      (X)| + \norm{\pi_{1} - \pi_{2}}
    \end{equation*}
    hence 
    \begin{equation}
      \label{eq:fpione:fpitwo}
      \norm{\IF_{\pi_{1}} - \IF_{\pi_{2}}} \lesssim \norm{\pi_{1} - \pi_{2}}. 
    \end{equation}

    Introduce        $\phi:\xR^{5}       \to        \xR$       given        by
    $2\phi(u) =  u_{1} |u_{2}+u_{3}| +  (1+u_{3})u_{4} + (1+u_{3})  u_{5}$ and
    $f_{1},    f_{2},   f_{4},    f_{5}$    be   the    function   given    by
    $f_{1}(o)  \equiv   (y  -  Q(a,x))/P(A=a|X=x)$,  $f_{2}   (o)  \equiv  x$,
    $f_{4} (o) \equiv  Q(1,x)$ and $f_{5}(o) \equiv  Q(-1,x)$.  Let $\xF_{1}$,
    $\xF_{2}$,   $\xF_{4}$,   $\xF_{5}$   be   the   singletons   $\{f_{1}\}$,
    $\{f_{2}\}$, $\{f_{4}\}$  and $\{f_{5}\}$, each  of them a  Donsker class.
    Let
    $\xtF  \equiv  \xF_{1}  \times  \xF_{2} \times  \Pi  \times\xF_{4}  \times
    \xF_{5}$  and note  that $\phi\circ  \bff =  \tIF_{\pi}$ if  $\bff\in\xtF$
    writes  as  $\bff=(f_{1},  f_{2},  \pi,   f_{4},  f_{5})$.   In  light  of
    \eqref{eq:key:two},      observe       now      that,       for      every
    $\bff_{1} = (f_{1}, f_{2}, \pi_{1}, f_{4}, f_{5}), \bff_{2}=(f_{1}, f_{2},
    \pi_{2}, f_{4}, f_{5}) \in \xtF$, it holds that
    \begin{equation*}
      |\phi \circ  \bff_{1} (o) -  \phi \circ \bff_{2} (o)|  \lesssim |\pi_{1}
      (x) - \pi_{2} (x)|
    \end{equation*}
    (the bound  on $Y$  implies that  $\norm{\gamma}_{\infty}$ is  finite). By
    \citep[][Theorem~2.10.6]{vanderVaartWellner1996},  whose   conditions  are
    obviously met, $\phi \circ \xtF  = \{\tIF_{\pi} : \pi\in\Pi\}$ is Donsker.
    Because $\Lambda  \equiv \{\xV(\pi) : \pi  \in \Pi\}$ (viewed as  a set of
    constant functions  with a  uniformly bounded  sup-norm) is  also Donsker,
    \citep[][Example~1.10.7]{vanderVaartWellner1996}        yields        that
    $\{\tIF_{\pi} - \lambda : \pi \in  \Pi, \lambda \in \Lambda\}$ is Donsker,
    and so is its subset $\xF$.
  \end{proof}

\begin{proof}[Proof of Theorem~\ref{thm:regDonsker}]\hypertarget{proof:thmregDonsker}
  In this proof, we make the dependence of $\pihat$ on $n$ explicit by writing
  $\pihat_n$.  By Lemma~\ref{lem:Pi:Donsker:F:Donsker},  $\Pi$ Donsker implies
  $\xF$ Donsker.   Consider the  empirical process $\xG_n$  as the  element of
  $\ell^{\infty}(\xF)$ characterized  by $\xG_n f \equiv  n^{1/2}(P_n-P)f$ for
  every $f\in\xF$.   We let $\xG_P\in\ell^{\infty}(\xF)$ denote  the zero-mean
  Gaussian        process        with        covariance        given        by
  $\E[\xG_P f \xG_P g] = Pfg - Pf Pg$.

  Since     $\Pi$      Donsker     implies     $\Pi$      totally     bounded,
  Lemma~\ref{lem:Pistarnonempty}    guarantees    the     existence    of    a
  $\pi^\star\in\Pibar^\star$.          For        any         $s>0$        and
  $\pi_s^\star\in  B_{\Pi}(\pi^\star,s)\subset  \Pi$,  (\ref{eq:valmax})  then
  \eqref{eq:valest} combined with \eqref{eq:nupione:nupitwo} yield in turn the
  first and second inequalities below:
  \begin{eqnarray}
    0\leq \Reg(\pihat_n)
    &=&  \xV^\star -  \xV(\pihat_n) =  [\xV(\pi_s^{\star}) -  \xV(\pihat_n)] +
        [\xV(\pi^{\star}) - \xV(\pi_s^{\star})] \nonumber \\ 
    &=&   (\xVhat  -   \xV)(\pihat_n)  -   (\xVhat  -   \xV)(\pi_s^{\star})  +
        [\xVhat(\pi_s^{\star}) - \xVhat(\pihat_{n})] + 
        [\xV(\pi^{\star}) - \xV(\pi_s^{\star})] \nonumber \\ 
    &\le&     (\xVhat-\xV)(\pihat_n)     -     (\xVhat-\xV)(\pi_s^\star)     +
          \left[\xV(\pi^\star)   -   \xV(\pi_s^\star)\right]   +   o_P(\Rem_n)
          \nonumber 
    \\ 
    \label{eq:fixeds:one}
    &\lesssim&               n^{-1/2}\left[\xG_n\IF_{\pihat_n}               -
               \xG_n\IF_{\pi_s^\star}\right] + s + 
               [1+o_P(1)]\Rem_n \\ 
    \label{eq:fixeds:two}
    &\leq&       2n^{-1/2}\sup_{\IF\in\xF}       |\xG_n \IF|        +       s       +
           [1+o_P(1)]\Rem_n. 
  \end{eqnarray}

  Since $z\mapsto \sup_{\IF  \in \xF} |z(f)|$ is continuous  and $\xF$ is
    Donsker,               the                continuous               mapping
    theorem~\citep[][Theorem~1.3.6]{vanderVaartWellner1996} implies  that the
  leftmost  term in  the  above  RHS sum  is  $O_{P} (n^{-1/2})$.   Therefore,
  \eqref{eq:fixeds:two} and $\Rem_{n} = o_{P} (n^{-1/2})$ imply
  \begin{equation*}
    0\leq  \Reg(\pihat_n) \lesssim  s +  O_{P} (n^{-1/2})  + [1+o_P(1)]  o_{P}
    (n^{-1/2}) 
  \end{equation*}
  where the random terms  do not depend on $s$. By letting $s$  go to zero, we
  obtain  $\Reg(\pihat_n) =  O_{P}  (n^{-1/2})$. The  remainder  of the  proof
  tightens this result to $\Reg(\pihat_n)= o_P(n^{-1/2})$.

  Let      us      go      back     to      (\ref{eq:fixeds:one}).       Since
  $\Rem_{n} = o_{P} (n^{-1/2})$, it also yields the tighter bound
  \begin{equation}
    \label{eq:firstresult}
    0       \leq       \Reg(\pihat_n)      \lesssim       n^{-1/2}\liminf_{s\downarrow
      0}\inf_{\pi^\star\in\Pibar^\star}\sup_{\pi_s^\star\in
      B_{\Pi}(\pi^\star,s)}\left[\xG_n\IF_{\pihat_n}-\xG_n\IF_{\pi_s^\star}\right]         +
    [1+o_P(1)] o_{P} (n^{-1/2}).  
  \end{equation}
  Let $g$  be defined as  in Lemma~\ref{lem:gsemicont}.  Note that  the second
  term  in   (\ref{eq:firstresult})  does  not  depend   on  $\pihat_n$.   Let
  $\{t_n\}_{n\geq  1}$   be  a  sequence   with  positive  values   such  that
  $t_{n}    \downarrow    0$.     As    $\pihat_n$    trivially    falls    in
  $\Pi_{\Reg(\pihat_n)+t_n}$,     we    can     take    a     supremum    over
  $\pi\in\Pi_{\Reg(\pihat_n)+t_n}$.      Multiplying     both     sides     of
  (\ref{eq:firstresult}) by $n^{1/2}$, we see that
  \begin{eqnarray*}
    0 \leq n^{1/2}\Reg(\pihat_n) 
    &\le&      \sup_{\pi\in\Pi_{\Reg(\pihat_n)+t_n}}     \xG_n\IF_{\pi}      -
          \limsup_{s\downarrow 0} \sup_{\pi^\star\in\Pibar^\star} \inf_{\pi\in
          B_{\Pi}(\pi^\star,s)} \xG_n(\IF_{\pi}) + o_P(1) \\ 
    &=& g(\xG_n,\Reg(\pihat_n)+t_n) + o_P(1).
  \end{eqnarray*}
  Above  we used  $\Reg(\pihat_n)+t_n$ rather  than $\Reg(\pihat_n)$  to avoid
  separately handling  the cases  where $\Pibar^\star\cap \Pi$  is and  is not
  empty.  
  
  The   conclusion  is   at  hand.    Recall  the   definition  of   $h$  from
  Corollary~\ref{lem:corollary:gsemicont}. Clearly the previous display yields
  the bounds
  \begin{equation}
    \label{eq:secondresult}
    0 \leq n^{1/2}\Reg(\pihat_n) \le h(\xG_n,\Reg(\pihat_n)+t_n) + o_P(1).
  \end{equation}
  We  have  already  established  that $\Reg(\pihat_n)  =  o_{P}  (1)$,  hence
  $0 < t_{n}  \leq \Reg(\pihat_n)+t_n = o_{P} (1)$ as  well.  Because $\xF$ is
  Donsker,  $\xG_n\leadsto  \xG_P$  in distribution  on  $\ell^{\infty}(\xF)$.
  Therefore,   $(\xG_n,   \Reg(\pihat_n)+t_n)   \leadsto  (\xG_{P},   0)$   in
  distribution on  $\ell^{\infty}(\xF) \times [0,\infty)$.  Almost  all sample
  paths  of  $\xG_P$  are  uniformly  continuous  on  $\xF$  with  respect  to
  $\norm{\cdot}$     \citep[][Section     2.1]{vanderVaartWellner1996},     so
  Corollary~\ref{lem:corollary:gsemicont}  applies   almost  surely   and  the
  continuous               mapping                theorem               yields
  $h(\xG_n,\Reg(\pihat_n)+t_n)   \leadsto  h(\xG_P,0)$   in  distribution   in
  $\ell^{\infty}(\xF)$.  This convergence also occurs in probability since the
  limit  is  almost  surely   constant  (zero).   By  (\ref{eq:secondresult}),
  $0\le \Reg(\pihat_n)=o_P(n^{-1/2})$. This completes the proof.
\end{proof}

\section{Additional Proofs}\label{app:generalERM}

\subsection{Sketch of Proof of Theorem~\ref{thm:generalRegDonsker}}

The proof of Theorem~\ref{thm:generalRegDonsker}  is very similar to that
  of Theorem~\ref{thm:regDonsker}. We only sketch  it and point out the places
  where they differ.

\begin{proof}[Sketch of Proof of Theorem~\ref{thm:generalRegDonsker}]\hypertarget{proof:thmgeneralRegDonsker}
  Obviously  (\ref{eq:valunifcons}) implies  that  $\Reg(\cdot)$ is  uniformly
  continuous on  $\overline{\Pi}$ with  respect to  $\norm{\cdot}$. Assumption
  (\ref{eq:Pitb})     then      shows     that     the      implication     of
  Lemma~\ref{lem:Pistarnonempty} also  holds in our context,  {\em i.e.}, that
  $\inf_{\pi\in\Pibar}\mathcal{R}(\pi)=0$  and  that   $\Pibar^\star$  is  not
  empty.   As $\Reg(\cdot)$  is  uniformly continous  and  the implication  of
  Lemma~\ref{lem:Pistarnonempty}        holds,       (\ref{eq:ca})        from
  Lemma~\ref{lem:cont}  also   holds.   Furthermore,   (\ref{eq:IFuc})  yields
  Lemma~\ref{lem:gsemicont}, for which Corollary~\ref{lem:corollary:gsemicont}
  remains    a    valid   corollary.     We    will    not   make    use    of
  Lemma~\ref{lem:Pi:Donsker:F:Donsker}:   we   will   instead   use   directly
  \eqref{eq:unifconv}.

  We    now   have    the   tools    needed   to    modify   the    proof   of
  Theorem~\ref{thm:regDonsker}. Using  the results we have  obtained thus far,
  and    replacing    (\ref{eq:valmax})     by    (\ref{eq:generalERM})    and
  (\ref{eq:valest})  by (\ref{eq:estseq}),  we see  that \eqref{eq:fixeds:two}
  from the proof of Theorem~\ref{thm:regDonsker} can be replaced by
  \begin{eqnarray}
    0\le \Reg(\widehat{\pi}_n)
    &\lesssim&  r_n^{-1} \left[\widetilde{\mathbb{G}}_n
               \IF_{\widehat{\pi}_n} - \widetilde{\mathbb{G}}_n
               \IF_{\pi_s^\star}\right]                       +
               \left[\xV(\pi^\star)-\xV(\pi_s^\star)\right]   +
               o_P(r_n^{-1}) \label{eq:fixeds:two:general} \\ 
    &\le&  2r_n^{-1}   \sup_{\IF\in\mathcal{F}}|\widetilde{\mathbb{G}}_n  \IF|  +
          \left[\xV(\pi^\star)-\xV(\pi_s^\star)\right] + o_P(r_n^{-1}). \nonumber 
  \end{eqnarray}
  By      (\ref{eq:unifconv})       and      the       continuous      mapping
  theorem~~\citep[][Theorem~1.3.6]{vanderVaartWellner1996},
  $\sup_{\IF\in\mathcal{F}}|\widetilde{\mathbb{G}}_n \IF|=O_P(1)$.  Hence, the
  leading term  in the  final inequality is  $O_P(r_n^{-1})$, where  this term
  does not  depend on  $s$. The  final term also  does not  depend on  $s$. By
  (\ref{eq:valunifcons}),   the   middle   term   above  goes   to   zero   as
  $s\downarrow   0$,    where   this   convergence   is    uniform   in   both
  $\pi^\star\in\overline{\Pi}^\star$                                       and
  $\pi_s^\star\in                 B_{\Pi}(\pi^\star,s)$.                 Thus,
  $\Reg(\widehat{\pi}_n)=O_P(r_n^{-1})$.

  We tighten  this result  to $\Reg(\widehat{\pi}_n)=o_P(r_n^{-1})$ as  in the
  proof  of  Theorem~\ref{thm:regDonsker}.   In particular,  nearly  identical
  arguments to those used in that proof show that
  \begin{equation*}
    0\le                      r_n                     \Reg(\widehat{\pi}_n)\le
    g(\widetilde{\xG}_n,\Reg(\widehat{\pi}_n)+t_n)         +         o_P(1)\le
    h(\widetilde{\xG}_n,\Reg(\widehat{\pi}_n)+t_n) + o_P(1). 
  \end{equation*}
  The  proof  concludes  by   noting  that  \eqref{eq:unifconv}  includes  the
  condition  that almost  all sample  paths of  $\widetilde{\mathbb{G}}_P$ are
  uniformly continuous  on $\mathcal{F}$  with respect to  $\norm{\cdot}$, and
  thus   the    right-hand   side   above   is    $o_P(1)$.    In   conclusion
  $\Reg(\widehat{\pi}_n)=o_P(r_n^{-1})$.
\end{proof}

\subsection{Proof for Section~\ref{subsec:discr:action}}
\label{app:proof:discr:action}

\begin{proof}[Proof of Lemma~\ref{lem:discr:action}]
  If  the bounded  class  $\Pi$ is  Donsker,  then it  is  totally bounded  in
  $L^2(P)$,  so \eqref{eq:Pitb}  is  met.  By  Lemma~\ref{lem:closurecompact},
  $\Pibar$  is  then compact.   Let  $\theta$  be  a Lipschitz  function  from
  $[-2,2]$  to  $[0,1]$  such  that  $\theta(0)  =  1$  and  $\theta(u)=0$  if
  $|u|\geq \min_{a\neq a' \in \xA}  |a-a'|$.  Introducing $\theta$ is merely a
  trick    to     generalize    Lemma~\ref{lem:VRcont}.      In    particular,
  \eqref{eq:key:one} becomes
  \begin{equation*}
    Q(\pi(X), X) = \sum_{a \in \xA} \theta(\pi(X) - a) Q(a,X)
  \end{equation*}
  for every policy $\pi:\xX \to \xA$, yielding
  \begin{equation}
    \label{eq:Nu:discr:gen}
    |\xV(\pi_{1}) - \xV(\pi_{2})| \lesssim \norm{\pi_{1} - \pi_{2}}
  \end{equation}
  for every $\pi_{1}, \pi_{2} \in \Pibar$, hence \eqref{eq:valunifcons}.  We
  also generalize \eqref{eq:key:two}, which becomes
  \begin{eqnarray}
    \label{eq:key:two:gen}
    \IF_{\pi}  (O)  =  \sum_{a \in  \xA}  \theta(\pi(X)-a)  \left(\theta(A-a)
    \frac{Y - Q(A,X)}{P(A|X)} + Q(a,X) \right) - \xV(\pi) 
  \end{eqnarray}
  for every $\pi \in \Pibar$.  The above equality and \eqref{eq:Nu:discr:gen}
  imply
  $\norm{\IF_{\pi_{1}} -  \IF_{\pi_{2}}} \lesssim \norm{\pi_{1}  - \pi_{2}}$
  for every $\pi_{1}, \pi_{2} \in \Pibar$, hence \eqref{eq:IFuc}. The second
  part of  the proof  of Lemma~\ref{lem:Pi:Donsker:F:Donsker} can  easily be
  generalized to  derive that  $\xF$ is Donsker  from \eqref{eq:key:two:gen}
  and the fact that $\Pi$  is itself Donsker. Therefore, \eqref{eq:unifconv}
  is met and the proof is complete. 
\end{proof}

\subsection{Proofs for Section~\ref{subsec:median}}
\label{app:proof:median}

We start by proving Lemma~\ref{lem:medianreguc}.

\begin{proof}[Proof of Lemma~\ref{lem:medianreguc}]
  Set  arbitrarily  $\pi_{1},\pi_{2} \in\Pibar$  and  $m  \in [\xV(\pi_2)  \pm
  c]$.     We     suppose     without      loss     of     generality     that
  $\mathcal{V}(\pi_1)\le  \mathcal{V}(\pi_2)$.  We  will  use  the  fact  that
  (\ref{eq:derivbd})                        implies                       that
  $F_{\pi_1}(\mathcal{V}(\pi_1))=F_{\pi_2}(\mathcal{V}(\pi_2))=1/2$    several
  times in this proof.

  Firstly, note that
  \begin{eqnarray}
    \notag
    \left|F_{\pi_1}(m)-F_{\pi_2}(m)\right|
    &=&        
        \left|\E\left[\frac{\Ind\{A=\pi_{1}(X)\}                             -
        \Ind\{A=\pi_{2}(X)\}}{P(A|X)} \Ind\{Y\le m\}\right]\right| \\
    \label{eq:first:ineq}
    &\le&              \frac{1}{2}\E\left[\frac{\left|\pi_{1}(X)             -
          \pi_{2}(X)\right|}{P(A|X)}      \Ind\{Y\le       m\}\right]      \le
          k_1\norm{\pi_{1}-\pi_{2}}
  \end{eqnarray}
  where $k_1$  is a finite, positive  constant that only depends  on the lower
  bound on $P(A|X)$ from the strong positivity assumption.

  Secondly,    the   continuous    differentiability   of    $F_{\pi_2}$   and
  (\ref{eq:derivbd}) imply  the existence  of $\tilde{m} \in  [m, \xV(\pi_2)]$
  such that
  \begin{eqnarray*}
    \left|F_{\pi_2}(\xV(\pi_2))-F_{\pi_2}(m)\right|
    &=& \left|\xV(\pi_2)-m\right| \times \dot{F}_{\pi_2}(\tilde{m})\\
    &\ge&    \left|\xV(\pi_2)-m\right|   \times    \inf_{\pi\in\Pibar}\inf_{m\in
          [\xV(\pi)\pm c]} \dot{F}_{\pi}(m).
  \end{eqnarray*}
  Therefore, there exists a finite, positive constant $k_2$ such that
  \begin{equation}
    \label{eq:second:ineq} 
    \left|\xV(\pi_2)-m\right|                                              \le
    k_2\left|F_{\pi_2}(\xV(\pi_2))-F_{\pi_2}(m)\right|. 
  \end{equation}

  The remainder  of this  proof is broken  into two parts:  we will  show that
  {\em(i)}           $\xV(\pi_2)-\xV(\pi_1)\le            c$           implies
  $\xV(\pi_2)-\xV(\pi_1)\le   k_1    k_2\norm{\pi_1-\pi_2}$,   and   {\em(ii)}
  $\norm{\pi_1-\pi_2}<  c/k_1 k_2$  yields  $\xV(\pi_2)-\xV(\pi_1)\le c$.   By
  combining    these    two    results,    we    will    thus    prove    that
  $\xV(\pi_2)-\xV(\pi_1)\le     k_1k_2     \norm{\pi_1-\pi_2}$     for     all
  $\norm{\pi_1-\pi_2}$ sufficiently  small, and  this will complete  the proof
  ($k_1k_2$ does not depend on $\pi_1$ or $\pi_2$).

  Recall                                                                  that
  $F_{\pi_1}(\mathcal{V}(\pi_1))=F_{\pi_2}(\mathcal{V}(\pi_2))=1/2$.        If
  $\xV(\pi_2)-\xV(\pi_1)\le  c$,  then   combining  \eqref{eq:first:ineq}  and
  \eqref{eq:second:ineq} at $m=\mathcal{V}(\pi_1)$ establishes {\em(i)}:
  \begin{eqnarray*}
    \xV(\pi_2)-\xV(\pi_1)
    &\le&               k_2\left[F_{\pi_2}(\mathcal{V}(\pi_2))               -
          F_{\pi_2}(\mathcal{V}(\pi_1))\right] \\ 
    &=&                k_2\left[F_{\pi_1}(\mathcal{V}(\pi_1))                -
        F_{\pi_2}(\mathcal{V}(\pi_1))\right] \le k_1k_2 \norm{\pi_1-\pi_2}.  
  \end{eqnarray*}
  We     argue     {\em(ii)}      by     contraposition.      Suppose     that
  $\mathcal{V}(\pi_1)<\mathcal{V}(\pi_2)-c$. By the monotonicity of cumulative
  distribution                                                      functions,
  $F_{\pi_2}(\mathcal{V}(\pi_1))\le
  F_{\pi_2}(\mathcal{V}(\pi_2)-c)$. Combining this with \eqref{eq:second:ineq}
  at $m=\mathcal{V}(\pi_2)-c$,
  \begin{equation*}
    F_{\pi_2}(\mathcal{V}(\pi_2))      -      F_{\pi_2}(\mathcal{V}(\pi_1))\ge
    F_{\pi_2}(\mathcal{V}(\pi_2))     -     F_{\pi_2}(\mathcal{V}(\pi_2)-c)\ge
    k_2^{-1} c. 
  \end{equation*}
  By        \eqref{eq:first:ineq}        and       the        fact        that
  $F_{\pi_2}(\mathcal{V}(\pi_2))=F_{\pi_1}(\mathcal{V}(\pi_1))=1/2$,  the  LHS
  expression    is   smaller    than   $k_1\norm{\pi_1-\pi_2}$.     Therefore,
  $\norm{\pi_1-\pi_2}\ge c/k_1k_2$.
\end{proof}

\begin{lemma}\label{lem:derivbd}
  Suppose the existence of a deterministic  $L>0$ such that, for all $m\in\xR$
  and sufficiently small $\epsilon>0$, with $P$-probability one,
  \begin{equation*}
    P(m<Y\le m+\epsilon|X) \leq L\epsilon.
  \end{equation*}
  Then,    for   all    $m    \in    \xR$   and    $\pi_{1},\pi_{2}\in\Pibar$,
  $|\dot{F}_{\pi_{2}}(m)            -            \dot{F}_{\pi_{1}}(m)|\lesssim
  \norm{\pi_{1}-\pi_{2}}$.
\end{lemma}

\begin{proof}[Proof of Lemma~\ref{lem:derivbd}]
  Set arbitrarily $\pi_{1},\pi_{2} \in\Pibar$ and $m \in \xR$.  In this proof,
  the   universal   positive   multiplicative  constants   attached   to   the
  $\lesssim$-inequalities  do not  depend  on $\pi_{1},  \pi_{2}, m$.   First,
  observe that
  \begin{equation*}
    \left|\dot{F}_{\pi_{2}}(m) - \dot{F}_{\pi_{1}}(m)\right|
    =                                               \lim_{\epsilon\downarrow
      0}\frac{1}{\epsilon}\left|F_{\pi_{2}}(m+\epsilon)  - F_{\pi_{2}}(m)  -
      F_{\pi_{1}}(m+\epsilon) + F_{\pi_{1}}(m)\right|. 
  \end{equation*}
  Second note that, for every $\epsilon>0$ small enough, 
  \begin{align*}
    \big|F_{\pi_{2}}(m+\epsilon)       -        F_{\pi_{2}}(m)       -
    &F_{\pi_{1}}(m+\epsilon) +    F_{\pi_{1}}(m)\big|\\ 
    &\le   \frac{1}{2}          \E\left[\frac{|\pi_{1}(X)-\pi_{2}(X)|}{P(A|X)}
      \Ind\{m<Y\le m+\epsilon\}\right] \\ 
    &\lesssim      \E\left[|\pi_{1}(X)-\pi_{2}(X)|     \Ind\{m<Y\le
      m+\epsilon\}\right] \\ 
    &=     \E\left[|\pi_{1}(X)      -     \pi_{2}(X)     |P(m<Y\le
      m+\epsilon|X)\right] \\ 
    &\le  L\epsilon\E\left[|\pi_{1}(X)-\pi_{2}(X)|\right]   \lesssim  \epsilon
      \norm{\pi_{1}-\pi_{2}}. 
  \end{align*}
  Returning to the first display completes the proof.
\end{proof}

\begin{corollary}\label{cor:derivcont}
  Under the  conditions of Lemmas~\ref{lem:medianreguc}~and~\ref{lem:derivbd},
  and    the    additional    assumption   (\ref{eq:contderiv}),    the    map
  $\pi\mapsto \dot{F}_\pi(\xV(\pi))$ is uniformly  continuous on $\Pibar$ with
  respect to $\norm{\cdot}$.
\end{corollary}

\begin{proof}[Proof                                                         of
  Corollary~\ref{cor:derivcont}]\hypertarget{proof:lemmedianreguc}
  Set  $\pi_{1},\pi_{2}\in\Pibar$.   By  Lemma~\ref{lem:medianreguc},  we  can
  choose  $\pi_{2}$  sufficiently  close   to  $\pi_{1}$  in  $L^2(P)$  (where
  ``sufficiently close'' does  not depend on the choice of  $\pi_{1}$) so that
  $|\xV(\pi_{1})-\xV(\pi_{2})|\le  c$,   where  $c$   is  the   constant  from
  (\ref{eq:derivbd}). For all such choices of $\pi_{1},\pi_{2}$,
  \begin{eqnarray*}
    \left|\dot{F}_{\pi_{2}}(\xV(\pi_{2}))-\dot{F}_{\pi_{1}}(\xV(\pi_{1}))\right|
    &\le&            \left|\dot{F}_{\pi_{2}}(\xV(\pi_{2}))           -
          \dot{F}_{\pi_{1}}(\xV(\pi_{2}))\right|                               +
          \left|\dot{F}_{\pi_{1}}(\xV(\pi_{2}))                              -
          \dot{F}_{\pi_{1}}(\xV(\pi_{1}))\right|     \\ 
    &\lesssim& \norm{\pi_{1}-\pi_{2}} + \omega(\norm{\pi_{1}-\pi_{2}}),
  \end{eqnarray*}
  where the  constant on the right  does not depend on  $\pi_{1},\pi_{2}$. The
  final  bound  used  Lemma~\ref{lem:derivbd}  and  (\ref{eq:contderiv}).   As
  $\omega(0)=0$  and   $\omega$  is  continuous   at  zero,  one   can  choose
  $\norm{\pi_{1}-\pi_{2}}$ sufficiently  small so that the  right-hand side is
  less than any $\epsilon>0$. As ``sufficiently small'' does not depend on the
  choice  of   $\pi_{1}$,  $\pi\mapsto  \dot{F}_\pi(\xV(\pi))$   is  uniformly
  continuous on $\Pibar$ with respect to $\norm{\cdot}$.
\end{proof}

{\singlespacing 

  \bibliographystyle{unsrtnat} \bibliography{fasterRatesForPolicyLearning}}

\end{document}